\documentclass[12pt]{amsart}
\usepackage{amsfonts,amsthm,amsmath,amssymb,mathrsfs}
\usepackage{algpseudocode,algorithm,algorithmicx,mathtools,thmtools}
\usepackage{bbm}
\usepackage{stmaryrd}
\usepackage[hidelinks]{hyperref} 
\usepackage{tikz,pgfplots,tikz-3dplot}
\usetikzlibrary{matrix, decorations.pathreplacing} 
\usepackage{todonotes,comment}

\newcommand{\alert}[1]{{\color{mycolor1}#1}}
\usepackage{paralist,enumitem}
\usepackage{bm}
\usepackage{textcomp}
\usepackage{vmargin}
\setpapersize{USletter}
\setmargrb{2.3cm}{1.5cm}{2.3cm}{1.5cm} 
\definecolor{mycolor1}{rgb}{0.00000,0.44700,0.74100}
\definecolor{mycolor2}{rgb}{0.8500, 0.3250, 0.0980}
\definecolor{mycolor3}{rgb}{0.9290, 0.6940, 0.1250}
\definecolor{mycolor4}{rgb}{0.4940, 0.1840, 0.5560}

\renewcommand{\emph}[1]{{\it #1}}

\newcommand{\ZZ}{\mathbb{Z}}
\newcommand{\QQ}{\mathbb{Q}}
\newcommand{\RR}{\mathbb{R}}
\newcommand{\CC}{\mathbb{C}}
\newcommand{\PP}{\mathbb{P}}

\newcommand{\calF}{{\mathcal F}}
\newcommand{\calFhat}{\hat{\mathcal F}}
\newcommand{\calG}{{\mathcal G}}
\newcommand{\calGhat}{\hat{\mathcal G}}
\newcommand{\calH}{{\mathcal H}}
\newcommand{\calHhat}{\hat{\mathcal H}}

\renewcommand{\P}{\mathcal{P}}
\newcommand{\U}{\mathcal{U}}
\renewcommand{\L}{\mathcal{L}}
\newcommand{\W}{\mathcal{W}}

\newcommand{\HH}{\mathscr H}
\newcommand{\A}{\mathscr A}
\newcommand{\OO}{\mathscr O}
\newcommand{\J}{\mathscr I}
\newcommand{\LL}{\mathscr L}
\newcommand{\Pol}{\mathscr P}

\newcommand{\cone}{\textup{cone}}

\DeclareMathOperator{\Vol}{\rm Vol}

\DeclareMathOperator{\im}{\rm im}

\DeclareMathOperator{\maxspec}{\rm MaxSpec}

\DeclareMathOperator{\conv}{\rm conv}

\DeclareMathOperator{\Hom}{\rm Hom}

\DeclareMathOperator{\Cl}{\rm Cl}
\DeclareMathOperator{\Pic}{\rm Pic}
\DeclareMathOperator{\MV}{\rm MV}
\DeclareMathOperator{\coker}{\rm coker}
\DeclareMathOperator{\val}{\rm val}

\newcommand{\ideal}[1]{\langle {#1} \rangle}
\newcommand{\pair}[1]{\langle {#1} \rangle}
\newcommand{\f}{\hat{f}}
\newcommand{\h}{\hat{h}}
\newcommand{\g}{\hat{g}}
\newcommand{\D}{\delta}
\newcommand{\Puis}{\CC \{\!\{\tau\}\!\} }
\newcommand{\Pow}{\CC [[\tau]] }
\newcommand{\z}{\zeta}

\newtheorem{thm}{Theorem}[section]
\newtheorem{prop}[thm]{Proposition}
\newtheorem{lem}[thm]{Lemma}
\newtheorem{cor}[thm]{Corollary}
\theoremstyle{definition}
\newtheorem{examplex}{Example}[section]
\newenvironment{example}
  {\pushQED{\qed}\examplex}
  {\popQED\endexamplex}
  \newtheorem{experimentx}{Experiment}
\newenvironment{experiment}
  {\pushQED{\qed}\experimentx}
  {\popQED\endexperimentx}
\newtheorem{conventionx}{Convention}
\newenvironment{convention}
  {\pushQED{\qed}\conventionx}
  {\popQED\endconventionx}
  
\newtheorem{remark}[thm]{Remark}

\definecolor{newgreen}{rgb}{0.0, 0.5, 0.0}
\definecolor{oldred}{rgb}{0.5, 0.0, 0.0}

\newcommand{\ffrac}[2]{\displaystyle\frac{#1}{#2}}

\newcommand{\hInc}{\widehat{\mathcal{V}}_{\A, L}}

\usepackage{ytableau}

\title{Polyhedral homotopies in Cox coordinates}
\author{T.~Duff}
\address{Tim Duff, School of Mathematics, 
            Georgia Institute of Technology, 686 Cherry St.~NW, Atlanta, GA 30308,  USA}
\email{tduff3@gatech.edu}
\urladdr{https://timduff35.github.io/timduff35/}
\author{S.~Telen}
\address{Simon Telen, Max Planck Institute for Mathematics in the Sciences, Leipzig, Germany}
\email{simon.telen@mis.mpg.de}
\urladdr{https://simontelen.webnode.com/}
\author{E.~Walker}
\address{Elise Walker, Department of Mathematics,
         Texas A\&M University, College Station, Texas 77843,  USA}
\email{walkere@math.tamu.edu}
\urladdr{http://www.math.tamu.edu/\~{}walkere}
\author{T.~Yahl}
\address{Thomas Yahl, Department of Mathematics,
         Texas A\&M University, College Station, Texas 77843,  USA}
\email{thomasjyahl@math.tamu.edu}
\urladdr{http://www.math.tamu.edu/\~{}thomasjyahl/}
\thanks{Research of Walker and Yahl supported in part by Simons Collaboration Grant for Mathematics 636314. }
\thanks{Work of Duff supported in part by the National Science Foundation under grant DMS-1719968.}

\begin{document}


\begin{abstract}
We introduce the Cox homotopy algorithm for solving a sparse system of polynomial equations on a compact toric variety $X_\Sigma$. The algorithm lends its name from a construction, described by Cox, of $X_\Sigma$ as a GIT quotient $X_\Sigma = (\CC^k \setminus Z)\sslash G$ of a quasi-affine variety by the action of a reductive group. 
Our algorithm tracks paths in the total coordinate space $\CC^k$ of $X_\Sigma$ and can be seen as a homogeneous version of the standard polyhedral homotopy, which works on the dense torus of $X_\Sigma$. It furthermore generalizes the commonly used path tracking algorithms in (multi)projective spaces in that it tracks a set of homogeneous coordinates contained in the $G$-orbit corresponding to each solution. 
The Cox homotopy combines the advantages of polyhedral homotopies and (multi)homogeneous homotopies, tracking only mixed volume many solutions and providing an elegant way to deal with solutions on or near the special divisors of $X_\Sigma$. In addition, the strategy may help to understand the deficiency of the root count for certain families of systems with respect to the BKK bound.
\end{abstract}

\maketitle

\section{Introduction}

In this paper, our aim is to solve systems of polynomial equations $\calFhat(x)=0,$ where $\calFhat=(\f_1(x), \ldots, \f_n(x))$ and $x=(x_1,\ldots , x_n)$. More specifically, we are concerned with the case where some isolated solutions of $\calFhat (x)=0$ lie near infinity. To make this precise, one must consider solutions in a suitable compactification of $\CC^n.$ Common choices for the compactification are the complex projective space $\PP^n$ or, more generally, a product of projective spaces $\PP^{n_1} \times \cdots \times \PP^{n_s}$ where $n_1+\cdots n_s=n.$ In each of these compactifications, the number of solutions to a generic system is a fixed number given by B\'{e}zout's theorem or its extension to a product of projective spaces. This observation gives the basis for numerical homotopy continuation methods---specifically, the {(homogeneous) total degree homotopy} and more general {multihomogeneous homotopies.} These may be used to find all isolated solutions to $\calFhat$ starting from a suitably generic {start system.} Conventionally, the solutions to the (multi)homogeneous start system all lie in $\mathbb{C}^n.$ To find solutions to $\calFhat$ at infinity, it is standard to use homogeneous coordinates and work in a generic {affine patch.}

In practice, infinite or nearly-infinite solutions to our {target system} $\calFhat$ present challenges for homotopy continuation. When tracking paths affinely, solutions of large magnitude are hard to estimate accurately. Even when homogeneous coordinates are used, these solutions are often singular or nearly-singular. Moreover, a large discrepancy between the number of start and target solutions may result in wasted computational resources.

The usual multihomogeneous start systems depend only on the degrees of $\f_1, \ldots, \f_n.$ In contrast, the {polyhedral homotopy} introduced in \cite{HS95,verschelde1994homotopies} takes the Newton polytopes of $\f_1, \ldots, \f_n$ into account.\footnote{Here, it is more natural to let $\f_i$ be Laurent polynomials, as we will do in later sections.} In this setting, the {polyhedral endgame}~\cite{huber1998polyhedral} may be used to detect solutions at infinity by numerically extrapolating coefficients of power series solutions. The number of start solutions for the polyhedral homotopy is given by the {BKK bound} from the celebrated Bernstein-Khovanskii-Kushnirenko (BKK) theorem~\cite{bernstein,kouchnirenko1976polyedres}.

In this paper, we propose the {Cox homotopy,} which combines salient features of the multihomogeneous and polyhedral homotopies. A schematic summarizing our approach is given in Figure~\ref{fig:schematic}. Our chosen compactification is an $n$-dimensional, compact, normal toric variety $X_{\Sigma },$ where the polyhedral fan $\Sigma$ refines the normal fan of each Newton polytope of $(\f_1, \ldots, \f_n).$ We may then regard a polynomial system as a section of a rank-$n$ vector bundle on $X_\Sigma .$ This section has a well-defined vanishing locus, which for a generic system consists of BKK-many isolated points (see Theorem~\ref{thm:bkk}).
\begin{figure}
\small
\begin{tabular}{c||c|c|c}
solution space & $\PP^n$                        & $\PP^{n_1} \times \cdots \times \PP^{n_s}$       & $X_\Sigma$                              \\ \hline
root count             & B\'ezout & multihomogeneous B\'ezout & BKK                              \\
graded ring            & $\CC[x_0, \ldots, x_n]$        & $\bigotimes_{i=1}^s \CC[x_{i0},\ldots,x_{in_i}]$ & $\CC[x_{\rho} ~|~ \rho \in \Sigma(1)]$   \\
homotopy               & homogeneous homotopy           & multihomogeneous homotopy                       & {Cox homotopy}                            
\end{tabular} \\
\normalsize
\caption{Schematic comparing (multi)homogeneous and Cox homotopies.}
\label{fig:schematic}
\end{figure}

This vanishing locus can be described globally by homogenizing the $\f_i$ to the {total coordinate ring} or {Cox ring} $S = \CC[x_\rho ~|~ \rho \in \Sigma(1)]$ of $X_\Sigma .$ The polynomial ring $S$, together with its grading by the divisor class group of $X_\Sigma$ and the irrelevant ideal, corresponds to the geometric construction of $X_\Sigma$ as a GIT quotient $X_\Sigma = (\CC^{k} \setminus Z) \sslash G$ of a quasi-affine variety by the action of an algebraic reductive group $G$ (much like the Proj-construction for $\PP^n$). Here we write $k$ for the number of rays in the fan $\Sigma$. This construction was described by Cox in \cite{cox1995homogeneous}. To represent points in $X_\Sigma,$ we use the global coordinates $x_\rho $ from the Cox construction, which we will refer to as {Cox coordinates}. In analogy with choosing an affine patch in (multi)homogeneous homotopies, we intersect the $G$-orbit corresponding to a point on $p \in X_\Sigma$ with a linear space of complementary dimension in $\CC^{k}$ to pick out finitely many sets of Cox coordinates representing $p$.

Our Cox homotopy offers the following advantages:
  \begin{itemize}
  \item[1)] By working in a compact space, we reduce the risk of prematurely truncating paths (see Experiments \ref{exp:computervision} and \ref{exp:weighted}).
  \item[2)] We are flexible in our choice of linear space. We may mitigate poor scaling or ill-conditioning by choosing a random linear space or the normal space to a $G$-orbit (see Experiment \ref{exp:hirzebruch}).
\item[3)] We retain the advantage of polyhedral homotopy that the BKK bound never exceeds any multihomogeneous B\'{e}zout bound, and may be substantially smaller (see Experiments \ref{exp:computervision}-\ref{exp:BS}).
\item[4)] Generalizing the multihomogeneous case, a solution at infinity lies on a {divisor} where some Cox coordinate equals $0.$ This can be used to heuristically establish when there are fewer finite solutions than the BKK bound, by investigating certain face systems (see Experiment \ref{exp:BS}).
  \end{itemize}

Although the Cox construction is a natural generalization of the familiar geometric quotient $ \PP^n = (\CC^{n+1} \setminus \{0\})/ (\CC \setminus \{0\})$, we believe the application to homotopy continuation is novel. Our work complements the recent use of Cox coordinates for dealing with non-toric solutions in a robust manner in numerical algebraic {normal form methods} \cite{telen2020thesis,telen2019numerical,bender2020toric}. Other closely-related work includes the aforementioned polyhedral endgame~\cite{huber1998polyhedral} and the use of toric compactifications in complexity analysis for sparse polynomial system solving~\cite{malajovich2019complexity,malajovich2020complexity}.

The paper is organized as follows.
In Section \ref{sec:prelim} we recall the necessary background on homotopy continuation, toric varities, and the Cox construction. Section \ref{sec:orbitdeg} describes the closure of $G$-orbits in $\PP^{k}$. We include a new combinatorial description of the degrees of these projective varieties (Proposition \ref{prop:orbitdeg}). These results are used in Section \ref{sec:mainthm}, where we state our main theorem (Theorem \ref{thm:coxhomotopy}) and we present the Cox homotopy algorithm (Algorithm \ref{alg:mainalg}), together with an algorithm which generalizes the {orthogonal patching} strategy proposed in \cite{hauenstein2018adaptive} for classical homogeneous homotopies (Algorithm \ref{alg:orthslice}). In Section \ref{sec:numexps} we give examples which demonstrate the advantages of our algorithms and compare the Cox homotopy to the polyhedral homotopy. 

\section{Preliminaries} \label{sec:prelim}

\subsection{Background on homotopy continuation}\label{subsec:homotopy}

A great variety of homotopy continuation methods exist, all of which adhere to the same setup---a system whose solutions are known (a \alert{start system}) is used to solve some other system (a \alert{target system}).
  Popular implementations include \texttt{Bertini}~\cite{Bertini}, \texttt{HOM4PS-3}~\cite{HOM4PS3}, \texttt{HomotopyContinuation.jl}~\cite{breiding2018homotopycontinuation}, \texttt{NAG4M2}~\cite{NAG4M2}, and \texttt{PHCpack}~\cite{PHCpack}. The numerical predictor-corrector methods lying at the core of homotopy continuation are treated at length in~\cite{AG90}. In the context of polynomial systems, the standard references are~\cite{Mor87, SW05, BertiniBook}.

For our purposes, a \alert{homotopy} $\calH(x;\tau )$ in variables $x=(x_1,\ldots , x_n)$ consists of polynomials $H_1 (x;\tau), \ldots , H_n (x; \tau)$ depending on an additional \alert{tracking parameter} $\tau$.
We use the convention that the start and target system are given by $\calG(x) = \calH(x,1)$ and $\calF (x) = \calH(x,0),$ respectively.
Thus, from any start solution $x_0\in \CC^n$ satisfying $\calG (x_0)=0,$ homotopy continuation seeks to numerically approximate solution paths $x(\tau )$ satisfying the initial-value problem
\begin{equation}
\label{eq:ivp}
  \begin{split}
  \ffrac{\partial \, \calH}{\partial \, x} \, \ffrac{\partial \, x}{\partial \, \tau} =
  - \ffrac{\partial \, \calH}{\partial \, \tau}, \qquad \qquad   x(1)=x_0. 
\end{split}
\end{equation}
In the terminology of \cite[Def.~4.5]{regen11}, a solution path $x(\tau)$ is said to be \alert{trackable} on an interval $I\subset [0,1]$ if the Jacobian $\partial \, \calH / \partial x$ is nonsingular at $(x(\tau), \tau)$ for all $\tau\in I.$
Typically, a randomization scheme such as the well-known \alert{gamma trick} is used so that $x (\tau )$ is trackable on $(0,1]$ with probability-one.
If $x(\tau)$ and $\tilde{x} (\tau)$ are distinct paths trackable on $I\subset [0,1],$ then the paths are disjoint in the sense that $x (\tau) \ne \tilde{x} (\tau )$ for any $\tau \in I.$ This follows from the inverse function theorem. 
  
A trackable path on $(0,1]$ converges to the \alert{endpoint} $\lim_{\tau \to 0^+} x(\tau ),$ provided that this limit exists, and otherwise \alert{diverges.}
To eliminate divergent solution paths, we may instead compute endpoints in some compactification of $\CC^n$. It is common practice to work in (multi)projective spaces. 

The first step for doing homotopy continuation in (multi)projective space is to \alert{homogenize} the equations in an appropriate way. We will use the following convention in the rest of this paper to distinguish between homogeneous and non-homogeneous notation. 
\begin{convention}
\label{convention:xt}
We denote by $t=(t_1,\ldots , t_n)$ a set of affine/toric variables and by $x=(x_1, \ldots , x_k)$ a set of homogeneous variables. Equations or systems of equations that are not homogeneous will be indicated by a circumflex, which is dropped after homogenization. For instance, homogenizing the system $\calFhat(t) = 0$ given by $\f_1(t)= \cdots = \f_n(t) = 0$, we obtain $\calF(x) = 0$ given by $f_1(x)= \cdots = f_n(x) = 0$.
\end{convention}

By homogenizing the equations $\f_1 = \cdots = \f_n = 0$ in an appropriate way, one is essentially extending the given relations on $\CC^n$ to relations on $\PP^{n_1} \times \cdots \times \PP^{n_s}$. The obtained equations have well-defined vanishing loci in $\PP^{n_1} \times \cdots \times \PP^{n_s}$ because of their homogeneity with respect to the appropriate multigrading. This is equivalent to reinterpreting each of the $\f_i$ as a section of a line bundle on (multi)projective space, such that the given equation corresponds to the trivialization of this section in one of the affine charts. The graded pieces of the multihomogeneous coordinate ring of $\PP^{n_1} \times \cdots \times \PP^{n_s}$ are in one-to-one correspondence with (isomorphism classes of) line bundles. For instance, for each $d \in \ZZ$, $\OO_{\PP^n}(d)$ is the line bundle on $\PP^n$ whose global sections are represented by homogeneous polynomials (in the standard sense) of degree $d$. Here's an example which shows the influence of the choice of compactification/homogenization on the solutions to the system. 

\begin{example}
\label{ex:motivating1}
Consider $\calFhat(t) = (\f_1(t),\f_2(t)) = 0,$ where
\begin{align*}  \label{eq:motivating1}
\f_1 &= 1 + t_1 + t_2 + t_1t_2 + t_1^2t_2 + t_1^3t_2, \qquad 
\f_2 = 2 + t_2 + t_1t_2 + t_1^2t_2. 
\end{align*}
The system $\calFhat(t)=0$ has three solutions in the algebraic torus $(\CC^*)^2\subset \CC^2$, given by $(t_1,t_2)  = (-1,-2), (e^{-\sqrt{-1} \frac{\pi}{3}}, -e^{\sqrt{-1} \frac{\pi}{3}}), (e^{\sqrt{-1} \frac{\pi}{3}}, -e^{-\sqrt{-1} \frac{\pi}{3}})$. We consider compactifying the solution space via the usual embeddings $\CC^2 \hookrightarrow \PP^2$ and $\CC^2 \hookrightarrow \PP^1 \times \PP^1.$

A \alert{total-degree} homotopy may be specified in homogeneous coordinates on $\PP^2$ by \[
\calH(x_0,x_1,x_2; \tau) = \gamma \, (1- \tau) \, \calF(x_0,x_1,x_2) + \tau \, \calG(x_0, x_1, x_2),
\]
where $\gamma \in \CC^*$ is a generic constant, and the homogeneous start and target equations are given by $\calG(x)=(x_1^4 - x_0^4, x_2^3-x_0^3)$ and $f_1 (x_0, x_1, x_2) = x_0^4 \, \f_1(x_1/x_0, x_2/x_0) \in \Gamma(\PP^n,\OO_{\PP^n}(4))$, $f_2 (x_0, x_1, x_2) = x_0^3 \, \f_3(x_1/x_0, x_2/x_0) \in \Gamma(\PP^n,\OO_{\PP^n}(3))$  respectively.
To get a unique representative for each point in $\PP^2,$ we may augment $\calH$ with a generic equation of the form $* \, x_0 + * \, x_1 + * \, x_2 = 1,$ representing an \alert{affine patch} on $\PP^2.$
There are $12$ start solutions.
Genericity of $\gamma $ and the patch implies that solution paths satisfying Equation~\eqref{eq:ivp} are trackable, and we may recover homogeneous representatives of the toric solutions.
The $9$ remaining endpoints are $(x_0:x_1:x_2) = (0:0:1)$ (multiplicity-$6$) and $(0:1:0)$ (multiplicity-$3$).

For solutions in $\PP^1 \times \PP^1,$ the system $\calFhat$ is homogenized in each variable separately.
Thus, the homogeneous target system becomes $ \calF = ( x_0^3 \, y_0 \, \f_1(x_1/x_0, y_1/y_0), x_0^2 \, y_0 \, \f_2(x_1/x_0, y_1/y_0)) = (f_1, f_2)$ with $f_1 \in \Gamma(\PP^1 \times \PP^1, \OO_{\PP^1 \times \PP^1}(3,1))$, $f_2 \in \Gamma(\PP^1 \times \PP^1, \OO_{\PP^1 \times \PP^1}(2,1))$ and we may work on a patch defined by $*\, x_0 + * \, x_1 = 1$ and $*\, y_0 + * \, y_1 = 1.$
The number of start solutions is now the corresponding multihomogeneous B\'{e}zout bound~\cite[Ch.~8, pp.~126--130]{SW05}, which in this case equals five.
Besides the toric solutions, the multihomogeneous homotopy recovers an additional solution at infinity: $(1:0) \times (0:1)$ with multiplicity 2.
\end{example}
Example~\ref{ex:motivating1} gives an instance of the \emph{family} $\calFhat(t_1,t_2;c) = (\f_1(t_1,t_2;c),\f_2(t_1,t_2;c))$ of systems with fixed monomial supports:
$$ \f_1 = c_{11} + c_{12} t_1 + c_{13} t_2 + c_{14} t_1t_2 + c_{15} t_1^2t_2 + c_{16} t_1^3t_2, \quad \f_2 = c_{21} + c_{22}t_2 + c_{23}t_1t_2 +c_{24} t_1^2t_2.$$
A member of this family with generic coefficients $c$ also has $3$ toric solutions. 
Moreover, in either the $\PP^2$-compactification or the $\PP^1 \times \PP^1$-compactification, the solutions at infinity remain exactly the same, respectively, for any generic choice of $c$.
Thus, it is natural to view these solutions at infinity as an artifact of the chosen compactification.

If we instead use the toric compactification $X_\Sigma$ proposed in this paper, then we may homogenize $(\f_1, \f_2)$ as described in Section~\ref{subsec:homogenization}. For generic choices of the $c_{ij}$, these homogenized equations only define 3 solutions on $X_\Sigma$, all contained in its dense torus. For this example, the toric variety $X_\Sigma$ turns out to be a Hirzebruch surface, as we will later see.

\subsection{Toric varieties and the Cox construction} \label{subsec:coxconstruction}
In this subsection we recall the construction of a toric variety as a GIT quotient. This construction was described by Cox in \cite{cox1995homogeneous}, and it is referred to as the \alert{Cox construction}. We should mention that the result was described earlier in the analytic category by Audin, Delzant and Kirwan, see \cite[Chapter 6]{audin2012topology} and references therein. For basic theory on toric varieties, the reader is referred to \cite[Chapters 1-4]{CLS} or \cite{fulton1993introduction}. Paragraph 5.5.1 in \cite{telen2020thesis} contains a short introduction with some worked out examples. 

Let $T = (\CC^*)^n$ be the \alert{algebraic torus} of dimension $n$. Its character and cocharacter lattices are denoted by $M = \Hom_\ZZ(T, \CC^*) \simeq \ZZ^n$ and $N = \Hom_\ZZ(M,\ZZ)$ respectively. A normal toric variety $X$ comes from a rational polyhedral fan $\Sigma$ in $N_\RR = N \otimes_\ZZ \RR$. We will denote the set of cones of dimension $d$ in $\Sigma$ by $\Sigma(d)$. We are interested in the case where $\Sigma$ is the \alert{normal fan} $\Sigma_\Pol$ of a full-dimensional lattice polytope $\Pol \subset M_\RR = M \otimes_\ZZ \RR$. Such fans are \alert{complete}, which means that $\bigcup_{\sigma \in \Sigma} \sigma = N_\RR$, and the corresponding toric variety $X$ is compact. We will sometimes denote $X = X_\Sigma$ to emphasize the correspondence between $X$ and its fan. 

Each $\rho_i \in \Sigma(1)$ has a unique primitive ray generator $u_i \in N$. It is convenient to collect the $u_i$ in a matrix 
$$ F = [u_1 ~ \cdots ~ u_k] \in \ZZ^{n \times k}.$$
In our context, the $u_i$ represent inward pointing \alert{facet normals} of the polytope $\Pol$. For this reason, we call $F$ the \alert{facet matrix}. This matrix represents a lattice homomorphism $F: N' \rightarrow N$, where $N' = \ZZ^k$, which is compatible with the fans $\Sigma'$ and $\Sigma$ in $N_\RR'$ and $N_\RR$ respectively. Here $\Sigma'$ is the fan of $\CC^k$ (i.e.\ the positive orthant in $\RR^k$ and all its faces), with some cones missing. It follows that $F$ gives a toric morphism $\pi: X_{\Sigma'} \rightarrow X_\Sigma$, where $X_{\Sigma'} = \CC^k \setminus Z$ and $Z$ is a union of coordinate subspaces. In this setting, the affine space $\CC^k$ is called the \alert{total coordinate space} and $Z$ is the \alert{base locus}. The base locus is defined by a monomial ideal $B$ in the coordinate ring $S = \CC[x_1,\ldots,x_k]$ of $\CC^k$:
$$ B = \left \langle \prod_{i \text{ s.t.\ } \rho_i \not \subset \sigma} x_i ~|~ \sigma \in \Sigma(n) \right \rangle \subset S \quad \text{and} \quad Z = V_{\CC^k}(B).$$
The ideal $B$ is also called the \alert{irrelevant ideal} of $S$.
Restricting the morphism $\pi$ to the torus $(\CC^*)^k$ we get the Laurent monomial map
\begin{equation} \label{eq:monmap}
\pi |_{(\CC^*)^k} = F \otimes_\ZZ \CC^* : (\CC^*)^k \rightarrow T
\end{equation}
where $(z_1,\ldots,z_k) \overset{\pi}{\mapsto} (z^{F_{1,:}}, \ldots, z^{F_{n,:}})$ (this uses the short notation $z^a = z_1^{a_i}\cdots z_k^{a_k}$ and $F_{i,:}$ for the $i$-th row of $F$). The kernel of $\pi |_{(\CC^*)^k}$ (as a group homomorphism) is a subgroup $G \subset (\CC^*)^k$ which acts on $\CC^k \setminus Z$ and the morphism $\pi$ is constant on $G$-orbits. The following theorem uses some terminology for GIT quotients from \cite[Section 5.0]{CLS}. 
\begin{thm}[\cite{cox1995homogeneous}] \label{thm:cox}
The morphism $\pi: \CC^k \setminus Z \rightarrow X_\Sigma$ coming from $F = [u_1 ~ \cdots ~ u_k]$ is an almost geometric quotient for the action of $G$ on $\CC^k \setminus Z$. Moreover, the open subset $U \subset X_\Sigma$ for which $\pi |_{\pi^{-1}(U)}$ is a geometric quotient is the largest simplicial toric subvariety of $X$, which is such that $(X_\Sigma \setminus U)$ has codimension at least 3 in $X_\Sigma$.
\end{thm}
Important for us is that this means that there is a \emph{very large} Zariski open subset $U \subset X_\Sigma$ such that $G$-orbits in $\pi^{-1}(U)$ are in one-to-one correspondence with points of $U$. The fan of $U$ is the subfan of $\Sigma$ consisting of all its simplicial cones.
\begin{example}
A familiar example is given by the construction of the projective space $X_\Sigma = \PP^n$ as a quotient of $\CC^{n+1} \setminus \{0\}$, in which case $\pi: \CC^{n+1} \setminus \{0\} \rightarrow \PP^n$ is given by $(x_1,x_2,\ldots, x_{n+1}) \mapsto(x_1:x_2:\cdots:x_{n+1})$. Since $\PP^n$ is simplicial, this quotient is geometric, i. e. $X_\Sigma = U$. 
\end{example}

To associate the ring $S$ (with its irrelevant ideal $B$) to our toric variety $X_\Sigma$, we equip it with a grading such that the vanishing locus of homogeneous elements in $\CC^k \setminus Z$ is stable under the action of $G$. The grading is by the \alert{divisor class group} $\Cl(X_\Sigma)$ of $X_\Sigma$, which is the group of Weil divisors modulo linear equivalence. For toric varieties, this group is easy to describe explicitly. Let $D_1, \ldots, D_k$ be the torus invariant prime divisors on $X_\Sigma$ corresponding to $\rho_1, \ldots, \rho_k$ respectively. We have the exact sequence
\begin{equation}\label{eq:Clseq}
0 \rightarrow M \overset{F^\top}{\longrightarrow} \bigoplus_{i=1}^k \ZZ \cdot D_i \overset{\P}{\longrightarrow} \Cl(X_\Sigma) \rightarrow 0,
\end{equation}
where the map $F^\top$ sends a character to its divisor (we will think of this map as a lattice map $\ZZ^n \rightarrow \ZZ^k$ given by the transpose of our matrix $F$) and $\P$ takes a torus invariant divisor to its class in $\Cl(X_\Sigma)$ (see \cite[Theorem 4.1.3]{CLS}). This shows that $\Cl(X_\Sigma) \simeq \ZZ^k / \im F^\top$ and every element of $\Cl(X_\Sigma)$ can be written as the class $[D]$ of some torus invariant divisor $D = \sum_{i=1}^k a_i D_i$. For an element $\alpha = [\sum_{i=1}^k a_i D_i] \in \Cl(X_\Sigma)$, we define the vector subspace 
$$ S_\alpha = \bigoplus_{F^\top m + a \geq 0} \CC \cdot x^{F^\top m + a},$$
where the sum ranges over all $m \in M$ satisfying $\pair{u_i,m} + a_i \geq 0$ (here $\pair{\cdot,\cdot}$ denotes the usual pairing between $N \simeq \ZZ^n$ and its dual $M \simeq \ZZ^n$). One can check that this definition is independent of the chosen representative for $\alpha$ and if $f = \sum_{F^\top m + a \geq 0} c_m x^{F^\top m + a} \in S_\alpha$, then for $g \in G \subset (\CC^*)^k$ we have
$$ f(g \cdot x) = \sum_{F^\top m + a \geq 0} c_m (g \cdot x)^{F^\top m + a} = g^a f(x).$$
It follows that $f \in S_\alpha$ has a well defined vanishing locus 
$$ V_{X_\Sigma}(f) = \{ p \in X_\Sigma ~|~ f(x) = 0 \textup{ for some } x \in \pi^{-1}(p) \}$$
and this definition extends trivially to $V_{X_\Sigma}(f_1, \ldots, f_s) = \bigcap_{i = 1}^s V_{X_\Sigma}(f_i)$ for elements $f_i \in S_{\alpha_i}$. An element $f \in S_\alpha$ is called \alert{homogeneous of degree $\alpha$}. The ring $S$, with its grading by $\Cl(X_\Sigma)$ and its irrelevant ideal $B$, is called the \alert{Cox ring} of $X_\Sigma$.
In conclusion, we summarize some terminology related to the Cox construction of $X_\Sigma$ in the table below. 
\[
\begin{matrix}
&\text{Algebra} & & \text{Geometry} \\ \hline
\text{Cox ring} &S = \bigoplus_{\alpha \in \Cl(X_\Sigma)} S_\alpha & \overset{\maxspec(\cdot)}{\longrightarrow} & \CC^k & \text{total coordinate space} \\
\text{irrelevant ideal}&B & \overset{V_{\CC^k}(\cdot)}{\longrightarrow} & Z & \text{base locus} \\
\text{class group} &\Cl(X_\Sigma) & \overset{\Hom_\ZZ(\cdot,\CC^*)}{\longrightarrow} & G & \text{reductive group}
\end{matrix}
\]

\subsection{Homogenization of sparse polynomial systems} \label{subsec:homogenization}
The standard homogenization used for sending polynomials to the multihomogeneous coordinate ring of $\PP^{n_1} \times \cdots \times \PP^{n_s}$ generalizes nicely for the compact toric varieties $X$ we consider in this paper and their Cox ring $S$. Line bundles on $X$ are in one-to-one correspondence with elements of the Picard group $\Pic(X) \subset \Cl(X)$, consisting of Cartier divisors modulo linear equivalence \cite[Chapter 4]{CLS}. Sections of these line bundles are homogeneous polynomials in $S$. 

The homogenization procedure we will use is described in detail in Section 3 of \cite{telen2019numerical}. We briefly recall how this works. Let $\f_1, \ldots, \f_n \in \CC[M] = \CC[t_1^{\pm 1}, \ldots, t_n^{\pm 1}]$ be a given set of Laurent polynomials and let $\A_1, \ldots, \A_n$ be their supports. That is, $\f_i$ can be written as $\f_i = \sum_{m \in \A_i} c_{i,m} t^m$ where $c_{i,m} \neq 0$. For $i = 1, \ldots, n$, let $\Pol_i = \conv(\A_i)$ be the Newton polytope of $\f_i$ (see for instance \cite[Chapter 7, \S 1]{cox2006using}).
The Minkowski sum of these polytopes is denoted by $\Pol = \Pol_1 + \cdots + \Pol_n$ and it is assumed to have dimension $n$. The normal fan $\Sigma_\Pol$ of $\Pol$ gives the toric variety $X = X_{\Sigma_\Pol}$. For each $i$, there is a canonical way of associating a nef torus invariant Cartier divisor $D_{\Pol_i} = \sum_{j=1}^k a_{i,j} D_i$ to $\Pol_i$. The class of this divisor in $\Pic(X)$ is denoted by $\alpha_i =[D_{\Pol_i}] \in \Pic(X)$. The vector space of Laurent polynomials with Newton polytope $\Pol_i$ is the vector space of sections of the vector bundle $\OO_X(\alpha_i)$ on $X$ \cite[Proposition 4.3.3]{CLS}. Moreover, this vector space can be identified with the degree $\alpha_i$ part of the Cox ring $S$ of $X$ \cite[Proposition 5.3.7]{CLS}. In summary, we have 
$$ \bigoplus_{m \in \Pol_i \cap M} \CC \cdot t^m \simeq \Gamma(X,\OO_X(\alpha_i)) \simeq S_{\alpha_i}.$$
The \alert{homogenization to the Cox ring} of $\f_i$ is given by the coefficients $a_{i,j}$ defining $D_{\Pol_i}$. Explicitly, we obtain the homogeneous polynomials $f_i$ from $\f_i$ by 
\begin{equation} \label{eq:homogenization}
\f_i = \sum_{m \in \A_i} c_{i,m} t^m \mapsto f_i = \sum_{m \in \A_i} c_{i,m} x^{F^\top m + a_i},
\end{equation}
where $F$ the facet matrix. Note that since $\A_i \subset \Pol_i \cap M$, we have that $F^\top m + a_i \geq 0$, so that the $f_i$ are indeed polynomials. 

We will think of the $n$-tuple $(f_1,\ldots,f_n)$ as a section of the rank $n$ vector bundle $\OO_X(\alpha_1) \oplus \cdots \oplus \OO_X(\alpha_n)$ on $X$. The zero locus of this section is the vanishing locus $V_X(f_1,\ldots,f_n)$. It contains the points defined by $\f_1=\cdots=\f_n=0$ in $T$, denoted by $V_T(\f_1,\ldots, \f_n)$.
Since $\OO_X(\alpha_1) \oplus \cdots \oplus \OO_X(\alpha_n)$ is a rank $n$ vector bundle on a variety of dimension $n$, the expected dimension of $V_X(f_1,\ldots,f_n)$ is 0. A well-known result by Bernstein, Khovanskii and Kushnirenko tells us how many points to expect. The statement involves the \alert{mixed volume} $\MV(\Pol_1,\ldots,\Pol_n)$ of the polytopes $\Pol_1,\ldots,\Pol_n$. 
\begin{thm}[BKK Theorem] \label{thm:bkk}
Let $f_i \in S_{\alpha_i}$ be as above. The variety $V_X(f_1,\ldots,f_n)$ contains at most $\MV(\Pol_1,\ldots, \Pol_n)$ isolated points on $X$. If $V_X(f_1,\ldots,f_n)$ is finite, then it consists of exactly $\MV(\Pol_1,\ldots, \Pol_n)$ points, counting multiplicities. For generic choices of the coefficients of the $f_i$, the number of roots in $T \subset X$ is exactly equal to $\MV(\Pol_1,\ldots,\Pol_n)$ and all roots have multiplicity one.
\end{thm}
\begin{proof}
See \cite[\S 5.5]{fulton1993introduction}.
\end{proof}
The number $\MV(\Pol_1,\ldots,\Pol_n)$ is often referred to as the \alert{BKK bound}. We are mainly interested in the case where $(f_1,\ldots,f_n)$ defines points \emph{outside of} $T$, that is, on the boundary $X \setminus T$ of the torus in $X$. 
To illustrate these examples of interest, we include the following example from~ \cite{telen2019numerical}. 
\begin{example} \label{ex:hirzebruch}
Consider the Laurent polynomials $\f_1, \f_2 \in \CC[t_1^{\pm 1}, t_2^{\pm 1}]$ given by 
\begin{align*}
\f_1 &= 1 + t_1 + t_2 + t_1t_2 + t_1^2t_2 + t_1^3t_2, \qquad 
\f_2 = 1 + t_2 + t_1t_2 + t_1^2t_2,
\end{align*}
which is equal to the system in Example \ref{ex:motivating1} up to the constant coefficient of $\f_2$. Although the BKK bound for the system $\calFhat = (\f_1,\f_2) = 0$ equals $\MV(\Pol_1,\Pol_2) =3$, the point $(-1,-1)$ is the unique solution (with multiplicity 1) in $T = (\CC^*)^2$. To explain this discrepancy with respect to the BKK bound, we extend the relations $\f_1 = \f_2 = 0$ to an appropriate toric variety.  
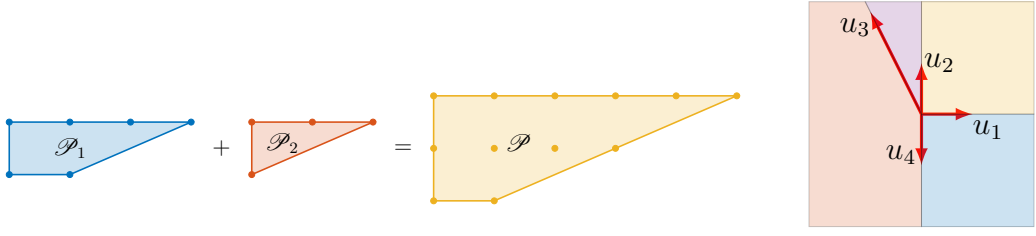
\begin{figure}
\centering
\begin{tikzpicture}[scale = 0.75]
\begin{axis}[%
width=5.500in,
height=1.100in,
scale only axis,
xmin=-0.5,
xmax=12.5,
xtick = \empty,
ymin=-1,
ymax=2,
ytick = \empty,
axis background/.style={fill=white},
axis line style={draw=none} 
]
\addplot [color=mycolor1,solid,thick, fill opacity = 0.2, fill = mycolor1,forget plot]
  table[row sep=crcr]{%
0	0\\
1	0\\
3	1\\
0	1\\
0	0\\
};
\addplot[only marks,mark=*,mark size=1.5pt,mycolor1
        ]  coordinates {
    (0,0) (1,0) (3,1) (2,1) (1,1) (0,1)
};

\node at (axis cs:1.0,0.5) {$\Pol_1$};
\node at (axis cs:3.5,0.5) {$+$};

\addplot [color=mycolor2,solid,thick, fill opacity = 0.2, fill = mycolor2,forget plot]
  table[row sep=crcr]{%
4	0\\
6	1\\
4	1\\
4	0\\
};
\addplot[only marks,mark=*,mark size=1.5pt,mycolor2
        ]  coordinates {
    (4,0) (6,1) (5,1) (4,1)
};
\node at (axis cs:4.5,0.6) {$\Pol_2$};
\node at (axis cs:6.5,0.5) {$=$};

\addplot [color=mycolor3,solid,thick, fill opacity = 0.2, fill = mycolor3,forget plot]
  table[row sep=crcr]{%
7	-0.5\\
8	-0.5\\
12	1.5\\
7	1.5\\
7	-0.5\\
};

\addplot[only marks,mark=*,mark size=1.5pt,mycolor3
        ]  coordinates {
    (7,-0.5) (8,-0.5) (7,0.5) (8, 0.5) (9, 0.5) (10, 0.5) (7, 1.5) (8,1.5) (9,1.5) (10,1.5) (11,1.5) (12,1.5)
};

\node at (axis cs:8.4,0.6) {$\Pol$};
\end{axis}
\end{tikzpicture}%
\quad 
\begin{tikzpicture}[scale=0.68]
\coordinate (O) at (0,0);
\coordinate (u1) at (1,0);
\coordinate (uu1) at (2.2,0);
\coordinate (u12) at (2.2,-2.2);
\coordinate (u2) at (0,-1);
\coordinate (uu2) at (0,-2.2);
\coordinate (u23) at (-2.2,-2.2);
\coordinate (uu23) at (-2.2,2.2);
\coordinate (u3) at (-1,2);
\coordinate (uu3) at (-1.1,2.2);
\coordinate (u4) at (0,1);
\coordinate (uu4) at (0,2.2);
\coordinate (u14) at (2.2,2.2);			
\draw[-latex, very thick, red] (O) -> (u1); 
\draw[-latex, very thick, red] (O) -> (u2);
\draw[-latex, very thick, red] (O) -> (u3);
\draw[-latex, very thick, red] (O) -> (u4); 
\draw[opacity=0.2,fill = mycolor1,] (O)--(uu1)--(u12)--(uu2)--cycle;
\draw[opacity=0.2,fill = mycolor2,] (O)--(uu2)--(u23)--(uu23)--(uu3)--cycle;
\draw[opacity=0.2,fill = mycolor3,] (O)--(uu4)--(u14)--(uu1)--cycle;
\draw[opacity=0.2,fill = mycolor4,] (O)--(uu3)--(uu4)--cycle;
\node at (1.3,-0.3) {$u_1$};
\node at (0.35,1.0) {$u_2$};
\node at (-1.3,1.7) {$u_3$};
\node at (-0.4,-0.8) {$u_4$};
\end{tikzpicture}
\caption{Polytopes and fan of the Hirzebruch surface from Example \ref{ex:hirzebruch}.}
\label{fig:hirzebruch}
\end{figure}
The polytopes and the fan are illustrated in Figure \ref{fig:hirzebruch}. The facet matrix $F$ is 
$$F = [u_1~u_2~u_3~u_4]= \begin{bmatrix}
					1&0&-1&0\\0&1&2&-1
					\end{bmatrix}.$$
The toric variety $X = X_\Sigma$ is the Hirzebruch surface $\HH_2$. The base locus in $\CC^4$ is given by $Z = V_{\CC^4}(x_1,x_3) \cup  V_{\CC^4}(x_2,x_4)$. The divisor $D_{\Pol_2}$ is  $D_{\Pol_2} = D_4$ (i.e.\ $a_{2,1} = a_{2,2} = a_{2,3} = 0, a_{2,4} = 1$, or $a_2 = (0,0,0,1)^\top$). The homogenization of the monomials $t^m$ in $\f_2$ is given by $F^\top m + a_2$: 
$$ F^\top \begin{bmatrix}
0 & 0 & 1 & 2 \\
0 & 1 & 1 & 1
\end{bmatrix} + \begin{bmatrix}
0&0&0&0\\0&0&0&0\\0&0&0&0\\1&1&1&1
\end{bmatrix} = \begin{bmatrix}
0 & 0 & 1 & 2\\ 0& 1&1&1\\0&2&1&0\\1&0&0&0
\end{bmatrix},$$
which gives 
$ f_2 = x_4 + x_2x_3^2 + x_1x_2x_3 + x_1^2x_3 \in S_{[D_4]}$.
Analogously we obtain
$ f_1 = x_3x_4 + x_1x_4 + x_2x_3^3 + x_1x_2x_3^2 + x_1^2x_2x_3 + x_1^3x_2 \in S_{[D_3 + D_4]}$.
The vanishing locus $V_X(f_1,f_2)$ on $X$ consists of three points, with Cox coordinates
$$z_1 = (-1,-1,1,1), ~ z_2 = (0,-1,1,1), ~ z_3 = (1,-1,0,1).$$
Hence, the relations $\f_1 = \f_2 = 0$ define three isolated points on $X$, which is the expected number. Note that $\pi(z_1)$ is the toric solution $(-1,-1)$ ($\pi$ denotes the quotient $\pi: \CC^4 \setminus Z \rightarrow X$) and the other solutions are on the boundary of the torus: $\pi(z_2) \in D_1, \pi(z_3) \in D_3$.
\end{example} 
 
\section{$G$-orbits in the Cox construction} \label{sec:orbitdeg}
Let $X = X_\Sigma$ be a compact toric variety corresponding to a complete fan $\Sigma$ and let $U \subset X$ be as in Theorem \ref{thm:cox}. As mentioned in the introduction, our homotopy algorithm will track a set of Cox coordinates for a point $p \in U \subset X$ by slicing the $G$-orbit $\pi^{-1}(p)$ with a linear space of complementary dimension. In order to understand what this dimension is and how many representatives 
there are in such a linear space,
this section is devoted to an explicit description of the dimension and degree of the projective closure of $G$-orbits $G \cdot z \subset \CC^k \setminus Z$ in the Cox construction. 

\subsection{Orbit parametrization}
Taking $\Hom_\ZZ(-,\CC^*)$ of the exact sequence \eqref{eq:Clseq}
gives the explicit description for $G$,
$$ G = \ker \Hom_\ZZ(F^\top, \CC^*) = \{g \in (\CC^*)^k ~|~ g^{F_{1,:}} = \cdots = g^{F_{n,:}} = 1 \},$$
 as a subgroup of $(\CC^*)^k$. Our first goal is to parametrize the orbit 
$$ G \cdot z = \{ g \cdot z = (g_1 z_1, \ldots, g_k z_k) \in (\CC^*)^k ~|~ g \in G \}$$
for $z \in \CC^k \setminus Z$.
Note that if $\Cl(X) \simeq \ZZ^{k-n}$ is free, then $G = \Hom_\ZZ(\Cl(X),\CC^*) \simeq (\CC^*)^{k-n}$ is a torus. In general, $G$ is a \alert{quasitorus} of dimension $k-n$, i.e.\ it is isomorphic to the direct sum of a torus and a finite abelian group. 

To find the orbit parametrization of $G \cdot z$, the first step is to compute the Smith normal form of the transpose of the facet matrix $F^\top \in \ZZ^{k \times n}$: 
$$ P F^\top Q = \text{diag}(s_1,\ldots,s_n),$$
where $\text{diag}(s_1,\ldots,s_n)$ is a diagonal matrix of size $k \times n$ with the invariant factors $s_i$ of $F^\top$ on its diagonal. We note that $s_i \ne 0$ since $F$ comes from a complete fan.
The submatrix $P''$ of $P$ containing the last $k - n$ rows of $P$ is a $\ZZ$-basis for the kernel of $F: \ZZ^k \rightarrow N$. We denote the submatrix of $P$ containing its first $n$ rows by $P'$. We have that 
$$ \Cl(X) \simeq \ZZ^k / \im ~ F^\top \simeq \ZZ/s_1 \ZZ \oplus \cdots \oplus \ZZ/s_n \ZZ \oplus \ZZ^{k-n} .$$
Note that, since the matrix $P$ is a representation of the map $\P$ from \eqref{eq:Clseq}, we can write this isomorphism explicitly as
\begin{equation} \label{eq:Pmap}
 [\sum_{i=1}^k a_i D_i ] \mapsto (\underbrace{(P'a)_1 + s_1 \ZZ, \ldots, (P'a)_n + s_n \ZZ}_{\ZZ/s_1 \ZZ \oplus \cdots \oplus \ZZ/s_n \ZZ}, \underbrace{ (P''a)_{1}, \ldots, (P''a)_{k-n}}_{\ZZ^{k-n}})
 \end{equation}
where $(P'a)_i$ is the $i$-th entry of the matrix-vector product $P'a$, and likewise for $P''a$. We also have that 
$$ G = \Hom_\ZZ(\Cl(X),\CC^*) = W_1 \oplus \cdots \oplus W_n \oplus (\CC^*)^{k-n}$$
where $W_i \subset \CC^*$ is the multiplicative group of $s_i$-th roots of unity.  
The inclusion $G \hookrightarrow (\CC^*)^k$ is the dual map $\P^\vee = \Hom_\ZZ(\P,\CC^*)$ of \eqref{eq:Pmap}, given by
%
\begin{equation} \label{eq:torusinclusion}
G \simeq \left( \bigoplus_{i=1}^n W_i \right ) \oplus (\CC^*)^{k-n} \rightarrow (\CC^*)^{k}, \quad (w,\lambda) \mapsto (w^{P_{:,1}'} \lambda^{P_{:,1}''}, \ldots, w^{P_{:,k}'} \lambda^{P_{:,k}''} ),
\end{equation}
where $P_{:,i}'$ denotes the $i$-th column of $P'$ and likewise for $P''$.
We conclude that $G \subset (\CC^*)^k$ is a union of tori isomorphic to
$$ T_{P''} = \{ (\lambda^{P_{:,1}''}, \ldots, \lambda^{P_{:,k}''} ) \in (\CC^*)^k ~|~ \lambda \in (\CC^*)^{k-n} \} \simeq (\CC^*)^{k-n}. $$
Moreover, the following statement follows immediately from this discussion. 
\begin{lem} \label{lem:orbitparam}
For $z \in \CC^k \setminus Z$, the orbit $G \cdot z$ of $z = (z_1, \ldots, z_k)$ is parametrized by the map 
\begin{equation} \label{eq:orbitparam}
\left( \bigoplus_{i=1}^n W_i \right ) \oplus (\CC^*)^{k-n} \rightarrow \CC^k \setminus Z \quad \text{given by} \quad (w,\lambda) \mapsto (w^{P_{:,1}'} \lambda^{P_{:,1}''}z_1, \ldots, w^{P_{:,k}'} \lambda^{P_{:,k}''}z_k ).
\end{equation}
\end{lem}
\begin{remark}
Note that the image of \eqref{eq:torusinclusion} is the orbit $G \cdot z$ of $z = (1,\ldots,1)$. 
\end{remark}
For a subset $\W \subset \CC^k \setminus Z$, we denote its Zariski closure in $\PP^k$ by
$$ \overline{\W} = \overline{ \{ (1: z_1: \cdots : z_k) \in \PP^k ~|~ (z_1,\ldots,z_k) \in \W \} } \subset \PP^k.$$
By applying Lemma~\ref{lem:orbitparam} to $\overline{G \cdot z}$, we immediately realize the following corollary. 
\begin{cor} \label{cor:onlydep}
For $z \in \CC^k \setminus Z$, the dimension and degree of the projective variety $\overline{G \cdot z}$ depend only on which $(\CC^*)^k$-orbit $z$ belongs to. Equivalently, they only depend on the set of indices $\J = \{i ~|~ z_i \ne 0 \}$.
\end{cor}
%

\subsection{Orbit dimension}

By Corollary \ref{cor:onlydep}, the dimension of $\overline{G \cdot z}$ is constant on the dense torus $(\CC^*)^k$ of $\CC^k \setminus Z$. In fact, it is constant on an even larger open subset of $X$ by some results from geometric invariant theory. 
\begin{lem} \label{lem:orbitdim}
For $z \in \pi^{-1}(U) \subset \CC^k \setminus Z$, where $U$ is as in Theorem \ref{thm:cox}, we have $\dim \overline{G \cdot z} = k-n$.  
\end{lem}
\begin{proof}
The set $\pi^{-1}(U) \subset \CC^k \setminus Z$ is the set of stable points for the action of $G$ on $\CC^k \setminus Z$ (these are the points $z$ whose stabilizer $G_z$ consists of finitely many points, see e.g.\ \cite[Proposition 1.26]{brion2010introduction} or \cite[Chapter 1, \S 4]{mumford1994geometric}). Therefore $\dim G \cdot z = \dim \overline{G \cdot z} = \dim G = k-n$.
\end{proof}
For points $z$ outside of $\pi^{-1}(U)$, the orbit $G \cdot z$ might not be closed in $\CC^k \setminus Z$, and a general formula for the dimension is given by $\dim G \cdot z = \dim G - \dim G_z$, where $G_z$ is the stabilizer of $G$ at $z$. The following example illustrates what may happen. 
\begin{example} \label{ex:pyramid}
Consider the toric threefold $X = X_\Sigma$ corresponding to a pyramid in $\RR^3$ whose normal fan $\Sigma$ has rays generated by the columns of
$$ F = \begin{bmatrix}
0 & 1 & 0 & -1 & 0 \\ 0 & 0 & 1 & 0 & -1 \\ 1 & -1 & -1 & -1 & -1 
\end{bmatrix} = \begin{bmatrix}
u_1 & u_2 & u_3 & u_4 & u_5
\end{bmatrix}.$$
The polytope and fan are illustrated in Figure \ref{fig:pyramid}. 
\begin{figure}[h!]
\centering
\tdplotsetmaincoords{70}{65}
\begin{tikzpicture}[baseline=(E.base),scale=1.5,tdplot_main_coords,scale = 0.8]
\coordinate (A) at (1.5,1.5,0);
\coordinate (B) at (1.5,-1.5,0);
\coordinate (C) at (-1.5,-1.5,0);
\coordinate (D) at (-1.5,1.5,0);
\coordinate (O) at (0,0,1.5);
\coordinate (E) at (0,0,1);

\draw[fill opacity=0.5,fill = mycolor2,] (A)--(B)--(C)--(D)--cycle;			
\draw[fill opacity=0.5,fill = mycolor2,] (O)--(C)--(D)--cycle;
\draw[fill opacity=0.5,fill = mycolor2,] (O)--(D)--(A)--cycle;
\draw[fill opacity=0.5,fill = mycolor2,] (O)--(A)--(B)--cycle;
\draw[fill opacity=0.5,fill = mycolor2,] (O)--(B)--(C)--cycle;
\end{tikzpicture}
\qquad \qquad
\begin{tikzpicture}[baseline=(O.base),scale=1.5,tdplot_main_coords, scale = 0.8]
\coordinate (O) at (0,0,0);
\coordinate (A) at (0,0,1.5);
\coordinate (B) at (1.5,0,-1.5);
\coordinate (C) at (0,1.5,-1.5);
\coordinate (D) at (-1.5,0,-1.5);
\coordinate (E) at (0,-1.5,-1.5);
																		
\draw[opacity=0.2,fill = mycolor1,] (O)--(A)--(C)--cycle;
\draw[opacity=0.2,fill = mycolor1,] (O)--(A)--(B)--cycle;
\draw[opacity=0.2,fill = mycolor1,] (O)--(A)--(D)--cycle;
\draw[opacity=0.2,fill = mycolor1,] (O)--(A)--(E)--cycle;
\draw[opacity=0.2,fill = mycolor1,] (O)--(B)--(C)--cycle;
\draw[opacity=0.2,fill = mycolor1,] (O)--(C)--(D)--cycle;
\draw[opacity=0.2,fill = mycolor1,] (O)--(D)--(E)--cycle;
\draw[opacity=0.2,fill = mycolor1,] (O)--(E)--(B)--cycle;
							
\draw[-latex] (O) -- (A) node[anchor=west] {$\rho_1$};
\draw[-latex] (O) -- (B) node[anchor=south] {$\rho_2$};
\draw[-latex] (O) -- (C) node[anchor=south] {$\rho_3$};
\draw[-latex] (O) -- (D) node[anchor=south] {$\rho_4$};
\draw[-latex] (O) -- (E) node[anchor=south] {$\rho_5$};
\end{tikzpicture} \qquad \qquad
\caption{Polytope and normal fan from Example \ref{ex:pyramid}.}
\label{fig:pyramid}
\end{figure}
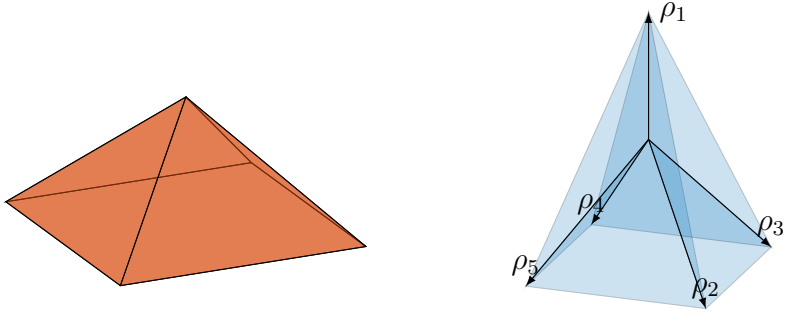
The irrelevant ideal is given by $B = \ideal{x_1,x_2x_3,x_2x_5,x_3x_4,x_4x_5}$. The corresponding two-dimensional base locus is $Z = V_{\CC^k}(x_1,x_2,x_4) \cup V_{\CC^k}(x_1,x_3,x_5)$. A Smith normal form computation yields 

\begin{figure}[H]
\centering
\begin{tikzpicture}[
    every left delimiter/.style={xshift=.4em},
    every right delimiter/.style={xshift=-.4em},
  ]
 \matrix (vec) [matrix of math nodes, left delimiter = {[}, right delimiter = {]}, inner sep=0pt, 
      nodes={inner sep=.3333em} ] {
0 & 1 & 0 & 0 & 0 \\ 0 & 0 & 1 & 0 & 0 \\ 1 & 0 & 0 & 0 & 0 \\ 2 & 1 & 0 & 1 & 0 \\ 2 & 0 & 1 & 0 & 1 \\ };
\node (a) at (vec-1-1.north) [left=10pt]{};
\node (b) at (vec-3-1.south) [left=10pt]{};
\node (c) at (vec-4-1.north) [left=10pt]{};
\node (d) at (vec-5-1.south) [left=10pt]{};
\node (F) at (vec-3-5) [right=10pt]{$F^\top$};
\draw [decorate, decoration={brace, amplitude=10pt}] (b) -- (a) node[midway, left=10pt] {\footnotesize $P'$};
\draw [decorate, decoration={brace, amplitude=6pt}] (d) -- (c) node[midway, left=10pt] {\footnotesize $P''$};
\matrix (m) at (vec-3-5) [right = 35pt, matrix of math nodes, left delimiter = {[}, right delimiter = {]}, inner sep=0pt, 
      nodes={inner sep=.3333em},] {
1 & 0 & 1 \\ 0 & 1 & 1 \\ 0 & 0 & 1 \\ }; 
\node at (vec-3-5) [right=80pt]{$=$};
\matrix (di) at (vec-3-5) [right = 100pt, matrix of math nodes, left delimiter = {[}, right delimiter = {]}, inner sep=0pt, 
      nodes={inner sep=.3333em},] {
1 & 0 & 0 \\ 0 & 1 & 0 \\ 0 & 0 & 1 \\ 0 & 0 & 0 \\ 0 & 0 & 0\\ };
\end{tikzpicture}
\end{figure}

\noindent such that by Lemma \ref{lem:orbitparam} the orbit of $(z_1,z_2,z_3,z_4,z_5)$ is parametrized by 
\begin{equation} \label{eq:parampyr}
(\lambda_1^2 \lambda_2^2 z_1, \lambda_1 z_2, \lambda_2 z_3, \lambda_1 z_4, \lambda_2 z_5), \quad (\lambda_1,\lambda_2) \in (\CC^*)^2.
\end{equation}
The subset $U \subset X$ for which $\pi^{-1}(U) \rightarrow U$ is geometric is the complement in $X$ of the torus invariant point $p \in X$ corresponding to $\sigma = \cone(\rho_2,\rho_3,\rho_4,\rho_5)$. It is clear that $z_{\{1\}} = (1,0,0,0,0) \in  \pi^{-1}(p)$. We see from Equation \eqref{eq:parampyr} that the orbit $G \cdot z_{\{1\}}$ has dimension 1. This is the unique closed $G$-orbit in $\pi^{-1}(p)$, and one can check that the stabilizer $G_{z_{\{1\}}}$ is one-dimensional. Since $\sigma$ is the smallest cone of $\Sigma$ containing $\rho_3$ and $\rho_5$, we see that the $(\CC^*)^5$-orbit 
$$(\CC^*)^{{\{1,2,4 \}}} = \{ (z_1,z_2,0,z_4,0) ~|~ (z_1,z_2,z_4) \in (\CC^*)^3 \} \subset \CC^5 \setminus Z$$
is contained in $\pi^{-1}(p)$. The same holds for $\rho_2$ and $\rho_4$. The fiber $\pi^{-1}(p)$ has dimension 3. From Equation~\eqref{eq:parampyr} it is clear that the orbit $G\cdot z$ has dimension 2 for $z \in (\CC^*)^{{\{1,2,4 \}}}$. Moreover, these orbits are not closed in $\CC^k \setminus Z$, as they contain $G \cdot {z_{\{1\}}}$ in their closure.
\end{example}

\subsection{Orbit degree}
For a finite set $\A = \{\alpha_0, \ldots, \alpha_s\} \subset \ZZ^\ell$ we use the notation $X_\A \subset \PP^{s}$ for the projective toric variety obtained as the closure of the image of $(\CC^*)^\ell \rightarrow \PP^{s}$, where $(t_1,\ldots,t_\ell) \mapsto (t^{\alpha_1}: \cdots : t^{\alpha_s})$. 
We associate to $\A$ the matrix $A$ of $\ell$-tuples in $\A$, i.e. $A = [\alpha_0 ~ \cdots ~ \alpha_s]: \ZZ^{s+1} \rightarrow \ZZ^\ell$.
One may compute the degree of the projective variety $X_\A$ using Kushnirenko's Theorem.
\begin{thm}[Kushnirenko's theorem \cite{kouchnirenko1976polyedres}] \label{thm:kush}
Let $X_\A \subset \PP^s$  be the projective toric variety defined by 
$$ A = \begin{bmatrix}
1 & \cdots & 1 \\ m_0 & \cdots & m_s 
\end{bmatrix} : \ZZ^{s+1} \rightarrow \ZZ^\ell.$$
Let $\Delta = \conv(m_0, \ldots, m_s) \subset \RR^{\ell-1}$ be the polytope obtained by taking the convex hull of the lattice points $m_0, \ldots, m_s \in \ZZ^{\ell -1}$ and suppose that $\Delta$ has dimension $\ell -1$. Then,
$$ \deg X_\A = \frac{(\ell-1)!}{q} \Vol(\Delta),$$
where $\Vol(\cdot)$ denotes the Euclidean volume and $q$ is the lattice index of $\im ([m_0 ~ \cdots ~ m_s]: \ZZ^{s+1} \rightarrow \ZZ^{\ell -1})$ in $\ZZ^{\ell-1}$. That is, for almost all choices of coefficients $c_{ij} \in \CC$, the system of equations 
$$ \sum_{j=0}^s c_{ij} \lambda^{m_j} = 0, \qquad i = \ell-1$$ 
has exactly $(\ell-1)! \Vol(\Delta)$ solutions $\lambda \in (\CC^*)^{\ell-1}$.
\end{thm}
With the notation of Subsection \ref{subsec:coxconstruction}, let $\J \subset \{1, \ldots, k\}$ be a subset of indices such that there is a cone $\sigma \in \Sigma$ containing each $\rho_i \in \Sigma(1), i \notin \J$. Equivalently, $\J$ is such that $\cone(e_i, i \notin \J) \in \Sigma'$, where $e_i$ is the $i$-th standard basis vector of $\RR^k$. The $(\CC^*)^k$-orbit $\{ z \in \CC^k \setminus Z ~|~ z_i \ne 0, \text{ for all } i \in \J \text{ and } z_i = 0 \text{ for all } i \notin \J \}$ is denoted by $(\CC^*)^{\J} \simeq (\CC^*)^{|\J|}$. We define the map $\P^\J$ as the restriction of $\P$ from \eqref{eq:Clseq} to $\bigoplus_{i \in \J} \ZZ \cdot D_i$ and consider the exact sequence
\begin{equation} \label{eq:seqorbitdeg}
0 \longrightarrow \ker \P^\J \longrightarrow \bigoplus_{i \in \J} \ZZ \cdot D_i \overset{\P^\J}{\longrightarrow} \Cl(X) \longrightarrow \coker \P^\J \longrightarrow 0.
\end{equation}
Taking duals shows that $(\P^\J)^\vee = \Hom_\ZZ(\P^\J,\CC^*) : G \rightarrow (\CC^*)^{{\J}}$ is given in coordinates by 
\begin{equation} \label{eq:orbitmap}
G \simeq \left( \bigoplus_{i=1}^n W_i \right ) \oplus (\CC^*)^{k-n} \rightarrow (\CC^*)^{{\J}} \simeq (\CC^*)^{|\J|}, \quad (w,\lambda) \mapsto (w^{P'_{:,i}}\lambda^{P''_{:,i}})_{i \in \J}.
\end{equation}
By Lemma \ref{lem:orbitparam}, the image of $(\P^\J)^\vee$ is the orbit $G \cdot z_\J$ of $z_\J = \sum_{i \in \J} e_i$, where $e_i$ is the $i$-th standard basis vector of $\CC^k$. The closure $\overline{G \cdot z_\J}$ is a union of a number of copies of 
the projective toric variety $X_{\A_\J}$ where $\A_\J$ is given by the columns of 
$$A_\J = \begin{bmatrix}
1 & 1 & \cdots & 1 \\
0 &  & P''_{:,\J}
\end{bmatrix},$$
and $P''_{:,\J}$ is the submatrix of $P''$ containing the columns indexed by $\J$.
\begin{example}
For $\J = \{1, \dots, k\}$, we have that $X_{\A_\J} = \overline{T_{P''}}$. Indeed,
$$X_\A = X_{\A_\J} := \overline{\left  \{ ( 1: \lambda^{P_{:,1}''}: \ldots: \lambda^{P_{:,k}''}) ~|~ \lambda \in (\CC^*)^{k-n} \right \}  } \subset \PP^k,$$
where $\A = \A_\J = \{\alpha_0, \ldots, \alpha_k \} \subset \ZZ^{k-n}$ is given by the columns of 
$$A = A_\J = \begin{bmatrix} 1 & 1 & \hdots & 1 \\ 0 &  & P'' \end{bmatrix}.$$ 
The dimension of $\overline{T_{P''}} = X_\A$ is $k-n$, by the fact that $P''$ has rank $k-n$ and  \cite[Proposition 2.1.2]{CLS}. The Gale transform of the lattice points in $\A$ consists of $\{u_1, \ldots, u_k, -\sum_{i=1}^k u_i \}$, where $u_i$ are the primitive ray generators of the rays in $\Sigma(1)$. This follows from the observation that 
$$ A \begin{bmatrix}
- \sum_{i=1}^k u_i^\top \\ F^\top 
\end{bmatrix} = 0. \qedhere$$
\end{example}
By Theorem \ref{thm:kush}, under the assumption that $P''_{:,\J}$ has rank $k-n$, we have that the degree of $X_{\A_\J}$ is $(k-n)!/q_\J \Vol(\Delta^\J)$, where $\Delta^\J = \conv( \{P''_{:,i}, i \in \J \} ) \subset \RR^{k-n}$ and $q_\J$ is the \alert{lattice index} of $\im P''_{:,\J}$ in $\ZZ^{k-n}$. We call the polytope $\Delta^\J$ the \alert{orbit polytope} corresponding to $\J$. 
\begin{prop} \label{prop:orbitdeg}
Let $\J \subset \{1, \ldots, k\}$ be a subset of indices such that $z_\J = \sum_{i \in \J} e_i \in \pi^{-1}(U)$, with $U$ as in Theorem \ref{thm:cox}. For all $z$ in the $(\CC^*)^k$-orbit $(\CC^*)^{\J} \subset \CC^k \setminus Z$, we have 
$$ \deg \overline{G \cdot z} = \frac{s_\J}{q_\J} (k-n)! \Vol(\Delta^\J),$$
where $q_\J, \Delta^\J$ are the lattice index and orbit polytope as above,  $s_\J$ is the product of the invariant factors of $\ker \P^\J$, and $\Vol(\cdot)$ denotes the Euclidean volume. In particular, $\overline{G \cdot z}$ is a union of $s_\J$ irreducible projective varieties of degree $\deg(\overline{G \cdot z})/s_\J$, each of which is equal to the projective toric variety $X_{\A_\J} \subset \CC^{\J} \subset \CC^k \setminus Z$ up to an invertible diagonal scaling.
\end{prop}
\begin{proof}
By Corollary \ref{cor:onlydep}, it suffices to show the proposition for $z = z_\J$. From \eqref{eq:orbitmap}, we see that $\overline{G \cdot z_\J}$ is a union of varieties obtained by scaling $X_{\A_\J}$. The dense torus of $X_{\A_\J}$, denoted by $T_{P''_{:,\J}}$, is the projection of $T_{P''}$ onto $(\CC^*)^\J$. The toric variety $X_{\A_\J}$ has dimension $k-n$ by Lemma \ref{lem:orbitdim} and degree  $(k-n)!/q_\J \Vol(\Delta^\J)$ by Theorem \ref{thm:kush}. By taking the dual of \eqref{eq:seqorbitdeg}, we see that $G \cdot z_\J$ is isomorphic to the quasitorus
$$ \Hom_\ZZ \left ( ~ \frac{\bigoplus_{i \in \J} \ZZ \cdot D_i}{\ker \P^\J} ~, \CC^* \right ), $$ 
which is a direct sum of a finite group of order $s_\J$ and the torus $T_{P''_{:,\J}}$.
\end{proof}
\begin{remark} \label{rem:genericorbit}
Note that when $\J = \{1, \dots, k\}$, $\ker \P^{\J} = \ker \P = \im F^\top$, such that the invariant factors of $\ker \P$ are $s_1, \ldots, s_n$ and $s = s_{\J} = \prod_{i=1}^n s_i$. Moreover, the lattice map $P'' = P''_{:,\J}$ is surjective, such that $q_{\J} = 1$.
\end{remark}
\begin{cor} \label{cor:freeclass}
If $\Cl(X) \simeq \ZZ^{k-n}$ is free, then for $\J$ as in \ref{prop:orbitdeg}, the degree of $\overline{G \cdot z} \subset \PP^k$ is $\deg \overline{G\cdot z} = \Vol(\Delta^{\J})$ for all $z$ in the $(\CC^*)^k$-orbit $(\CC^*)^{\J}$. Moreover, $\overline{G \cdot z} \subset \PP^k$ is a(n irreducible) toric variety. 
\end{cor}

Proposition \ref{prop:orbitdeg} provides a direct way of computing $\deg(\overline{G \cdot z})$ for any $z \in \pi^{-1}(U)$. We implemented this in Macaulay2 \cite{M2}. The code will be made available at \url{https://mathrepo.mis.mpg.de}. The following two examples illustrate applications of Proposition~\ref{prop:orbitdeg}.

\begin{example}
The \emph{double pillow} is the toric variety $X = X_\A \subset \PP^4$ where $\A$ consists of the columns of $A = \begin{bmatrix}
1&1&1&1&1\\1&0&0&-1&0\\0&1&0&0&-1
\end{bmatrix}$ \cite[Subsection 3.3]{sottile2017ibadan}.
The fan corresponding to the double pillow has facet matrix $ F= \begin{bmatrix}
1&1&-1&-1\\1&-1&-1&1
\end{bmatrix}$
with Smith normal form 
$$ \begin{bmatrix}
0&-1&0&0\\-1&1&0&0\\0&1&0&1\\1&0&1&0
\end{bmatrix} 
F^\top \begin{bmatrix}
0 & -1\\ -1&-1
\end{bmatrix}
= \begin{bmatrix}
1&0\\0&2\\0&0\\0&0
\end{bmatrix}.$$
This shows that the class group has torsion: $ \Cl(X) \simeq \ZZ/2\ZZ \oplus \ZZ^2$ via 
$$[\sum_{i=1}^4 a_i D_i] \mapsto (a_2-a_1 + 2\ZZ, a_2+a_4,a_1+a_3).$$ The reductive group $G$ is isomorphic to $\{-1,1\} \oplus (\CC^*)^2  $ via
$(w, \lambda_1,\lambda_2) \mapsto(w^{-1}\lambda_2, w \lambda_1, \lambda_2, \lambda_1)$. The closure is a union of two planes in $\PP^4$, and every orbit closure $\overline{G \cdot z}$ for $z \in (\CC^*)^k$ is equal to $\overline{G}$ up to a diagonal change of coordinates. Therefore, these orbits have degree 2. The orbit polytope $\Delta^\J$ is the standard simplex in $\RR^2$.
\end{example}

\begin{example} \label{ex:hirzebruch2}
Consider again the Hirzebruch surface $\HH_2$ from Example \ref{ex:hirzebruch}. The Smith normal form of the facet matrix $F^\top$ is given by 
$$ \begin{bmatrix}
-1&0&0&0\\
0&-1&0&0\\
-1&2&-1&0\\
0&-1&0&-1
\end{bmatrix} F^\top \begin{bmatrix}
1&0\\0&1
\end{bmatrix} = \begin{bmatrix}
-1&0\\0&-1\\0&0\\0&0
\end{bmatrix}, \quad \text{hence} \quad P'' = \begin{bmatrix}
-1&2&-1&0\\
0&-1&0&-1
\end{bmatrix}.$$
The class group $\Cl(\HH_2)$ is free, $G \simeq (\CC^*)^2$ and the orbit of $z = (z_1, \ldots, z_4) \in \CC^4 \setminus Z$ is parametrized by 
\begin{equation} \label{eq:orbithirz}
 G \cdot z = \{ (z_1 \lambda_1^{-1}, z_2 \lambda_1^2 \lambda_2^{-1}, z_3 \lambda_1^{-1}, z_4 \lambda_2^{-1}) ~|~ (\lambda_1,\lambda_2) \in (\CC^*)^2 \}.
 \end{equation}
The orbit polytopes are shown in Figure \ref{fig:orbitpolytopes}. 
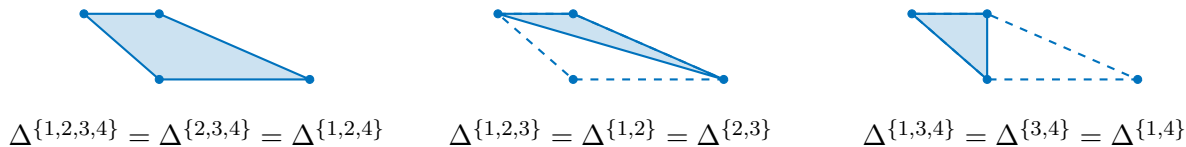
\begin{figure}
\centering
\begin{tikzpicture}
\begin{axis}[%
width=6.500in,
height=1.100in,
scale only axis,
xmin=-1.5,
xmax=15,
xtick = \empty,
ymin=-1.2,
ymax=2,
ytick = \empty,
axis background/.style={fill=white},
axis line style={draw=none} 
]
\addplot [color=mycolor1,solid,thick, fill opacity = 0.2, fill = mycolor1,forget plot]
  table[row sep=crcr]{%
0	1\\
1	1\\
3	0\\
1	0\\
0	1\\
};
\addplot[only marks,mark=*,mark size=1.5pt,mycolor1
        ]  coordinates {
    (0,1) (1,0) (3,0) (1,1) 
};
\node at (axis cs:1.5,-0.8) {\small $\Delta^{\{1,2,3,4\}}  = \Delta^{\{2,3,4\}} = \Delta^{\{1,2,4\}} $};

\addplot [color=mycolor1,dashed,thick, fill opacity = 0.0, fill = mycolor1,forget plot]
  table[row sep=crcr]{%
5.5	1\\
6.5	1\\
8.5	0\\
6.5	0\\
5.5	1\\
};
\addplot [color=mycolor1,solid,thick, fill opacity = 0.2, fill = mycolor1,forget plot]
  table[row sep=crcr]{%
5.5	1\\
6.5	1\\
8.5	0\\
5.5	1\\
};
\addplot[only marks,mark=*,mark size=1.5pt,mycolor1
        ]  coordinates {
    (5.5,1) (6.5,0) (8.5,0) (6.5,1) 
};
\node at (axis cs:7,-0.8) {\small $ \Delta^{\{1,2,3\}} = \Delta^{\{1,2\}} =  \Delta^{\{2,3\}} $};

\addplot [color=mycolor1,dashed,thick, fill opacity = 0.0, fill = mycolor1,forget plot]
  table[row sep=crcr]{%
11	1\\
12	1\\
14	0\\
12	0\\
11	1\\
};
\addplot [color=mycolor1,solid,thick, fill opacity = 0.2, fill = mycolor1,forget plot]
  table[row sep=crcr]{%
11	1\\
12	1\\
12	0\\
11	1\\
};
\addplot[only marks,mark=*,mark size=1.5pt,mycolor1
        ]  coordinates {
    (11,1) (12,0) (14,0) (12,1) 
};
\node at (axis cs:12.5,-0.8) {\small $ \Delta^{\{1,3,4\}} = \Delta^{\{3,4\}} = \Delta^{\{1,4\}} $};

\end{axis}
\end{tikzpicture}%
\caption{Orbit polytopes from Example \ref{ex:hirzebruch2}.}
\label{fig:orbitpolytopes}
\end{figure}
By Corollary \ref{cor:freeclass}, if $z \in (\CC^*)^3$, $\overline{G \cdot z}$ has degree $\Vol(\Delta^{\{1,2,3,4\}}) = 3$. If $z \in (\CC^*)^{\{1,3,4\}} \cup  (\CC^*)^{\{1,2,3\}} $, the degree drops to $\Vol(\Delta^{\{1,3,4\}}) = \Vol(\Delta^{\{1,2,3\}}) = 1$. 
\end{example}

\section{Cox homotopies: coefficient-parameter theory and algorithms } \label{sec:mainthm}

As in the previous section, $X = X_\Sigma$ is an $n$-dimensional toric variety such that $\pi: \CC^k \setminus Z \to X$ is an almost geometric quotient which is constant over $G$-orbits on $\CC^k \setminus Z$.
Moreover, on the dense open subset $U \subset X$, $\pi_{\pi^{-1}(U)}: \pi^{-1}(U) \rightarrow U$ is geometric. 

This section describes our homotopy algorithm for solving polynomial systems on $X$ by tracking points in the total coordinate space $\CC^k$.
For any $p\in U$, we will slice the $G$-orbit $\pi^{-1}(p)$ with a general linear space of dimension $n$ to get degree-many representatives for $\pi^{-1}(p)$.
Theorem~\ref{thm:coxhomotopy} establishes that the orbits in our homotopy remain disjoint.
Thus, we may track only one representative per orbit when the target system is generic in the sense of Theorem~\ref{thm:bkk}.
In Section~\ref{subsec:endgame}, we consider the case of a non-generic target system.
This presents a subtlety not encountered in the multihomogeneous case---namely, some paths in the total space $\CC^k$ may either diverge or converge to a point in the base locus (see Example~\ref{ex:hirzebruch3}).
We propose an endgame in Section~\ref{subsec:endgame} for finding a path whose endpoint represents a point in $X.$
Pseudocode for the main algorithm and subroutines are in Section~\ref{subsec:algorithms}.
Finally, in Subsection~\ref{subsec:orthslice} we outline the orthogonal patching strategy considered in Experiment~\ref{exp:hirzebruch}.

\subsection{Coefficient-parameter homotopy in the dense torus}\label{subsec:mainthm}
 
As our homotopy is akin to a coefficient-parameter homotopy, we consider the \emph{family} of all systems with fixed supports $\A = (\A_1, \ldots, \A_n)$ in $M$ (cf.~Subsection~\ref{subsec:homogenization}). Thus, we consider $\h_i (t;c) = \sum_{m \in \A_i} c_{i,m} t^m, \, i = 1, \ldots, n,$ where the parametric coefficients $c_{i,m}$ may vary. We denote the affine space of parameters by $\CC^\A$. We assume hereafter that the Minkowski sum $\Pol = \Pol_1 + \cdots + \Pol_n$, where $\Pol_i = \conv(\A_i),$ has full dimension $n,$ so that the BKK number $\MV(\Pol_1,\ldots,\Pol_n)$ from Theorem~\ref{thm:bkk} is positive.

The homotopy considered in our main algorithm (Algorithm~\ref{alg:mainalg}) is given by
\begin{equation}
  \label{eq:phiHomotopyHomog}
\left(\calH \left(x; \tau\right), \L (x)\right) = 0, 
\end{equation}
where $\L(x) = Ax + b$ determines an (affine) linear space $L = \L^{-1}(0) =  \{ x \in \CC^k \mid \L (x)=0 \},$ the path $\phi : [0,1] \to \CC^{\A}$ is assumed to be smooth, and $\calH$ is obtained by homogenizing
\begin{equation}
  \label{eq:phiHomotopyTorus}
\calHhat (t; \tau) = \left(\h_1\left(t; \phi (\tau)\right), \ldots ,\h_n\left(t; \phi (\tau)\right)\right) = 0.
\end{equation}
We also need that $\phi ,$ $A,$ and $b$ satisfy certain genericity conditions as outlined in Theorem~\ref{thm:coxhomotopy}.
We will refer to such homotopies as \alert{Cox homotopies.} 

Theorem~\ref{thm:coxhomotopy} gives the basis for Algorithm~\ref{alg:mainalg}. To prove this theorem, it is natural to consider the incidence variety defined for fixed $L$ by
\begin{equation}
\label{eq:hIncDef}
    \hInc = \left\{ (x, c) \in \left( \left( \CC^k \setminus Z \right) \cap L \right) \times \CC^\A \mid \calF (x;c) =  0 \right\}.
\end{equation}

Solution paths for the Cox homotopy in Equation~\eqref{eq:phiHomotopyHomog} correspond to \emph{lifts} of the parameter path $\phi $ to the variety $\hInc .$
For generic $\phi ,$ $A,$ and $b,$ Theorem~\ref{thm:coxhomotopy} implies that there are exactly $d \D $ such paths, where $d$ denote the degree of the $G$-orbit closure $\overline{G \cdot (e_1 + \cdots + e_k)}.$
Note that $d$ may be computed using Remark~\ref{rem:genericorbit}.

\begin{thm} \label{thm:coxhomotopy}
There exist Zariski open subsets $\U_\A \subset \CC^\A$ and $\U_{A, b} \subset \CC^{(k-n) \times k} \times \CC^{k-n}$ such that for all $c^* \in \U_\A$ and all $(A,b) \in \U_{A, b}$, with $\L(x) = Ax + b$, we have 
\begin{enumerate}
\item $V_T(\calHhat(t;c^*)) = \{\z_1, \ldots, \z_\D \}$ consists of $\D$ points, \label{item:mt1}
\item $V_{(\CC^*)^k}(\calH(x;c^*), \L(x)) = \{ z_{ij} ~|~ i = 1, \ldots, \D, j = 1, \ldots, d \}$ consists of $d \D$ points,  \label{item:mt2}
\item for some labeling of these points, $\pi(z_{ij}) = \z_i$ for all $i,j$. \label{item:mt3}
\end{enumerate}
For fixed $(A,b) \in \U_{A,b}$ and a smooth function $\phi: [0,1] \rightarrow \U_\A$, 
\begin{enumerate}[resume]
\item the homotopy $\calHhat(t;\phi(\tau)) = 0$ has $\D$ trackable paths $\zeta_i : [0,1] \to (\CC^*)^n, i = 1, \dots, \D,$\label{item:mt4}
\item the homotopy $(\calH(x;\phi(\tau)), \L (x)) = 0$ has $d \D$ trackable paths $z_{i j} : [0,1] \to (\CC^*)^k,$

$i = 1, \ldots, \D, j = 1, \ldots, d$,\label{item:mt5}
\item after relabeling, we have that $\pi\circ z_{ij} = \z_i$ for all $i,j$. \label{item:mt6}
\end{enumerate}
\end{thm}

\begin{proof}
The variety $V_{(\CC^*)^k}(\calH(x;c^*))$ is a union of $G$-orbits, which have dimension $k-n$ by Lemma \ref{lem:orbitdim}.
We take $\U_{A, b}$ such that $G \cdot (e_1 + \cdots + e_k) \cap L$ always consists of $d$ points for $L = \L^{-1}(0)$.
Theorem~\ref{thm:bkk}, together with Remark~\ref{rem:genericorbit}, implies that the coordinate projection $\pi_\A: \hInc \to \CC^\A$ is a dominant map with generically finite fibers.
More precisely, $\pi_\A$ is a branched cover of degree $d \D,$ meaning that the fiber $\pi_{\A}^{-1} (c)$ consists of $d \D $ points for all $c\in \U_\A,$ where $\CC^\A \setminus \U_\A$ denotes the \alert{branch locus} of $\pi_\A.$
The set $\U_\A$ is Zariski open by~\cite[Theorem 2.29]{shafarevich1994basicag}.
We also observe that $\pi_\A$ has a factorization induced by $\pi$. That is, we may write $\pi_\A = \pi_\A^1 \circ \pi_\A^2$ where $\pi_\A^1$ and $\pi_\A^2$ are branched covers of degree $d , \D,$ respectively.
So far we have shown (1)--(3).
Now, whenever $(z,c)\in \pi_{\A}^{-1} (c)$ and $c\in \U_\A,$ the derivative $D_{z,c} \, \pi_{\A}$ is an invertible linear map by~\cite[Theorems 2.30]{shafarevich1994basicag}.
This gives (5).
Similarly, (4) follows for $\zeta = \pi (z)$ by considering $D_{\zeta , c} \, \pi_\A^1$ and applying the chain rule.
For (6), we simply note that $(A,b) \in \U_{A,b}$ implies there are exactly $d$ paths $z_{i1}, \ldots , z_{id}$ for each $\zeta_i$ satisfying $\zeta_i = \pi \circ z_{ij}.$
\end{proof}

\begin{remark} \label{rem:movinglinspace}
It is possible to allow the linear space $L$ to change as $\tau$ moves from $1$ to $0$. In other words, we could consider a Cox homotopy given by $(\calH(x;\tau), \L_\tau(x)) = 0$ where $L(\tau) = \L_\tau^{-1}(0)$ is a sufficiently generic, continuous path in the Grassmannian $G(n,k)$. In Subsection \ref{subsec:orthslice} we will exploit this observation to propose the strategy of \alert{orthogonal slicing} for Cox homotopies.
\end{remark}

\begin{remark} \label{rem:straightline}
Since the complement of $\U_\A$ is of complex codimension at least 1 in $\CC^\A$, the image of a general 1-real-dimensional parameter path $\phi: [0,1] \rightarrow \CC^\A$ with $c_0 = \phi(0) \in \U_\A$, $c_1 = \phi(1) \in \U_\A$ is contained in $\U_\A$. A standard trick for obtaining such a general path is by setting $\phi(\tau) = (1- \tau) c_0 + \gamma \tau c_1$, where $\gamma \in \CC$ is a random constant. This is known as the \alert{gamma trick}, see for instance \cite[Lemma 7.1.3]{SW05}. In this case, the Cox homotopy $(\calH, \L)$ is given by $\calH(x;\tau) = (1-\tau) \calG + \gamma \tau \calF$, where $\calF(x) = \calH(x;0)$ and $\calG(x) = \calH(x;1)$. If the start system $(\calG, \L)$ is suitably generic, we can set $\gamma = 1$. Homotopies obtained in this way are often called \alert{straight line homotopies} \cite[Sec 2.1]{BertiniBook}.
\end{remark}

An important consequence of Theorem \ref{thm:coxhomotopy} is that for parameter paths $\phi: [0,1] \rightarrow \U_\A$ it is sufficient to track only $\D = \MV(P_1,\ldots,P_n)$ paths in the Cox homotopy $(\calH,\L)$ to track the paths defined by $\calH(x;\tau)$ on $X$. Indeed, for $i = 1, \ldots, \D$, it suffices to track only one of the paths $z_{ij}(\tau), j = 1, \ldots, d$ for $\tau$ going from 1 to 0, since all of these will land on a representative of the same $G$-orbit when $\tau = 0$.

\subsection{Degenerating orbits and a specialized endgame}\label{subsec:endgame}

For a set $\A = (\A_1, \ldots, \A_n)$ of supports, we fix a smooth parameter path $\phi: [0,1] \rightarrow \CC^\A$ and consider the corresponding homotopy $\calHhat(t;\tau)$. We denote $\calFhat(t) = \calHhat(t;0),~\calGhat(t) = \calHhat(t;1)$ and $\calF, \calG,\calH$ for the associated homogeneous counterparts. We assume that $\calGhat$ defines $\D = \MV(\Pol_1, \ldots, \Pol_n)$ isolated solutions in $T \subset X$, but $\calFhat$ does not. That is, $|V_T(\calFhat)|< \D$.

In this setup, if $\phi$ is suitably generic, Theorem \ref{thm:coxhomotopy} tells us that $\calHhat (t; \tau), ~ \tau \in (0,1]$ defines $\D$ disjoint paths $\tau \mapsto \z_i(\tau), i = 1, \ldots, \D$ in the dense torus $T $. In this section we investigate what happens for $\tau \rightarrow 0$. Since $|V_T(\calFhat)|< \D$ by assumption, at least one of these paths, say $\zeta_i(\tau)$, moves out of the torus $T$ for $\tau \rightarrow 0$. By compactness of $X \supset T$, the endpoint $\lim_{\tau \rightarrow 0^+} \zeta_i(\tau)$ exists in $X$ and it satisfies 
$$ \z_i(0) = \lim_{\tau \rightarrow 0^+} \zeta_i(\tau) \in X \setminus T .$$
We assume that $\z_i(0) \in U$ is an isolated point of $V_X(\calF)$, where $U$ is as in Theorem \ref{thm:cox}. 
We first show that in an analytic neighborhood $\U \subset \CC$ of $\tau = 0$, the $G$-orbit $\pi^{-1}(\z_i(\tau))$ has a representative $r(\tau) = (r_1(\tau), \ldots, r_k(\tau))$ given by Puiseux series. We use the notation $\Pow$ for the power series ring, $\Puis$ for the field of Puiseux series, and $\val: \Puis \rightarrow \QQ$ for its standard valuation. As in Section \ref{sec:orbitdeg}, we let $(\CC^*)^\J$ be the $(\CC^*)^k$-orbit $\{ z \in \CC^k \setminus Z ~|~ z_i \ne 0, \text{ for all } i \in \J \text{ and } z_i = 0 \text{ for all } i \notin \J \}$ and work with index sets $\J$ such that $(\CC^*)^\J \subset \pi^{-1}(U)$.
\begin{lem} \label{lem:representative}
Let $\tau \mapsto \z_i(\tau)$ be a solution path of $\calHhat(t;\tau)$ such that $\z_i(0) \in U$ is an isolated point in $V_X(\calF)$, where $U$ is as in Theorem \ref{thm:cox}. In a neighborhood $\U \subset \CC$ of $\tau = 0$, there is an algebraic function $r: \U \rightarrow \CC^k \setminus Z,~ r \in \Puis^k$ such that $\pi^{-1}(\z_i(\tau)) = G \cdot r(\tau)$ for all $\tau \in \U$. Moreover, if $\J$ is such that $\z_i(0) \in \pi((\CC^*)^\J)$, the valuations of $r(\tau) = (r_1(\tau), \ldots, r_k(\tau))$ satisfy
\begin{equation*}
\val(r_i(\tau)) > 0,~ i \notin \J \qquad \text{and} \qquad \val(r_i(\tau)) = 0, ~i \in \J.
\end{equation*}
\end{lem}
\begin{proof}
Since $\z_i(0) \in U \subset X$, $\overline{\pi^{-1}(\z_i(0))}$ has dimension $k-n$ (Lemma \ref{lem:orbitdim}) and for almost all choices of the linear space $L = \L^{-1}(0)$, $L \cap \overline{\pi^{-1}(\z_i(0))}$ consists of isolated points. For such a fixed choice of $L$, pick a point $r_0 \in L \cap \overline{\pi^{-1}(\z_i(0))} \subset V_{\CC^k \setminus Z}(\calF,\L)$. Since $\z_i(0)$ is an isolated point in $V_X(\calF)$, $r_0$ is an isolated solution of $(\calF(x),\L(x)) = (\calH(x;0), \L(x))$ in $\CC^k \setminus Z$ and there is an open neighborhood $\U$ of $\tau = 0$ and an algebraic function $r: \U \rightarrow \CC^k \setminus Z$ satisfying $r(0) = r_0$ and $(\calH(r(\tau);\tau), \L(r(\tau))) = 0$. The statement about the valuations follows immediately from $r(0) \in \pi^{-1}(\z_i(0)) \subset (\CC^*)^\J$.
\end{proof}

\begin{remark} \label{rem:power}
By reparametrizing, e.g.\ by setting $\tau \leftarrow \tau^m$ for a positive integer $m$, we may assume that $r(\tau)$ in Lemma \ref{lem:representative} is given by a power series. 
\end{remark}

Take $\tau^* \in \U$ where $\U$ is as in Lemma \ref{lem:representative} and suppose that $r(0) \in (\CC^*)^\J$. Our aim is to formalize the following intuition. A $G$-orbit $G \cdot r(\tau^*) \subset \CC^k \setminus Z \subset \PP^k$ for $r(\tau^*) \in (\CC^*)^k$ hits the general linear space $L$ in $d$ points, where $d = s (k-n)! \Vol(\Delta^{\{1,\ldots,k\}})$ (see Remark \ref{rem:genericorbit}). When we move $r(\tau^*)$ towards $r(0) \in (\CC^*)^{\J}$, the points in $(G \cdot r(\tau^*)) \cap L$ move in $L$. By Proposition \ref{prop:orbitdeg}, as the representative $r(\tau)$ enters $(\CC^*)^{\J}$ the orbit degree might drop, meaning that possibly $| (G \cdot r(\tau^*) ) \cap L | > | (G \cdot r(0) ) \cap L |$. The `missing' points in $(G \cdot r(0) ) \cap L$ may be due to some points traveling into $(L \cap Z)$, some points traveling to infinity out of the total coordinate space, or two different points colliding at $\tau = 0$. We are interested in answering the following question: how many of the points in $(G \cdot r(\tau^*) ) \cap L$ eventually land in $ (G \cdot r(0) ) \cap L$? Equivalently, if $r(\tau)$ is a representative for $\pi^{-1}(\z_i(\tau))$, how many of the representative paths $z_{ij}(\tau)$ in $\CC^k \setminus Z$ travel to the orbit $G \cdot r(0)$ as $\tau \rightarrow 0$?

\begin{thm} \label{thm:boundary}
Let $\U \subset \CC$ be an open neighborhood of $ 0 \in \CC$ and let $r(\tau) = (r_1(\tau), \ldots, r_k(\tau))$ be a map $\U \rightarrow \CC^k \setminus Z$ such that $r_i(\tau) \in \Pow \subset \Puis$. Let $\J$ be such that $r(0) \in (\CC^*)^\J$ and suppose that 
\begin{equation} \label{eq:valuations}
\val(r_i(\tau)) > 0,~ i \notin \J \qquad \text{and} \qquad \val(r_i(\tau)) = 0, ~i \in \J.
\end{equation}
For almost all affine maps $\L(x) = Ax + b$, there are $s (k-n)! \Vol (\Delta^\J)$ $k$-tuples $z(t) = (z_1(\tau), \ldots, z_k(\tau)) \in \Pow^k$ such that 
$$\val(z_i(\tau)) > 0,~ i \notin \J \qquad \text{and} \qquad \val(z_i(\tau)) = 0, ~i  \in \J$$
 and $ z(\tau) \in G \cdot r(\tau),~A~z(\tau) +b = 0,~\text{for $\tau$ in a neighborhood $\U' \subset \U$ of $0 \in \CC$.}$ 
\end{thm}
\begin{proof}
The proof uses some notation from Section \ref{sec:orbitdeg}. In particular, we recall that $P$ is one of the unimodular matrices in the Smith normal form of the transposed facet matrix $F^\top$ and $P''$ contains its last $k-n$ rows. We first consider the case where $\Cl(X)$ is torsion free. By Lemma \ref{lem:orbitparam} it suffices to show that for almost all choices of $A,b$, there are $(k-n)! \Vol (\Delta^\J)$ $(k-n)$-tuples $\lambda(\tau) = (\lambda_1(\tau), \ldots, \lambda_{k-n}(\tau)) \in \Puis^{k-n}$ such that $\val(\lambda_i(\tau)) = 0, i = 1, \ldots, k-n$ and 
$$ A (r_1(\tau) \lambda(\tau)^{P''_{:,1}}, \ldots, r_k(\tau) \lambda(\tau)^{P''_{:,k}} )^\top + b = 0,$$
or $\sum_{j=1}^k A_{ij} r_j(\tau) \lambda(\tau)^{P''_{:,j}} + b_i = 0, i = 1, \ldots, k-n$. 
The solutions $\lambda(\tau)$ with $\val(\lambda_i(\tau)) = 0, i = 1, \ldots, k-n$ are given by 
$$\lambda(\tau) = (\ell_1, \ldots, \ell_{k-n}) + \text{ higher order terms},$$
where $(\ell_1, \ldots, \ell_{k-n}) \in (\CC^*)^{k-n}$ is a solution of 
\begin{equation} \label{eq:faceteq}
\sum_{j\in \J} A_{ij} r_j(0) \lambda^{P''_{:,j}} + b_i = 0, \quad i = 1, \ldots, k-n.
\end{equation}
Indeed, by \eqref{eq:valuations} the terms where $j \in \J$ correspond to the facet with facet normal $(0, \ldots, 0, 1)$ on the lower hull of the lifted point set $\{ (P''_{:,j}, \val(r_j(\tau)) ) \}_{j=1,\ldots,k} \subset \RR^{k-n+1}$, see e.g.\ \cite[Section 2]{HS95}. 
By Theorem \ref{thm:kush} and the fact that $r_j(0) \neq 0$ for all $j\in \J$, \eqref{eq:faceteq} has $(k-n)!\Vol(\Delta^\J)$ solutions for almost all choices of $A, b$. 

It remains to show that each of these solutions $\lambda(\tau)$ gives a different $k$-tuple $z(\tau)$ given by $(r_1(\tau) \lambda(\tau)^{P''_{:,1}}, \ldots, r_k(\tau) \lambda(\tau)^{P''_{:,k}})$. For $\varepsilon > 0$, let $\U_\varepsilon = \{ \tau \in \CC ~|~ |\tau| < \varepsilon \}$ and note that there exists $\varepsilon > 0$ such that for any two solutions $\lambda(\tau)$ and $\mu(\tau)$, $\lambda(\tau^*) \neq \mu(\tau^*)$ for all $\tau^* \in \U_\varepsilon$  and $r_i(t^*) \neq 0$ for all $\tau^* \in \U_\varepsilon \setminus \{ 0\}$. Since $P'': \ZZ^k \rightarrow \ZZ^{k-n}$ is surjective, this implies 
$$( r_1(\tau^*) \lambda(\tau^*)^{P''_{:,1}}, \ldots, r_k(\tau^*) \lambda(\tau^*)^{P''_{:,k}} ) \neq  ( r_1(\tau^*) \mu(\tau^*)^{P''_{:,1}}, \ldots, r_k(\tau^*) \mu(\tau^*)^{P''_{:,k}} ) , \quad \text{for } \tau^* \in \U_{\varepsilon} \setminus \{0\}.$$

Using an analogous argument, we see that in the case where $\Cl(X)$ has torsion, each irreducible component of $G \cdot r(\tau)$ contributes $(k-n)!\Vol(\Delta^\J)$ $k$-tuples $z(\tau)$. 
\end{proof}

\begin{example}\label{ex:hirzebruch3}
We continue Example \ref{ex:hirzebruch2} by showing how the orbit $G\cdot z$ `degenerates' from a degree 3 surface to a plane as $z$ moves into $D_4$. We consider the equations defining the closure $\overline{G \cdot z}$ in $\PP^4$. It can be seen from \eqref{eq:orbithirz} that for $z \in (\CC^*)^4$,
\begin{equation} \label{eq:orbiteq}
 \overline{G \cdot z} = V_{\PP^4}(z_3x_1 - z_1x_3, ~ z_4x_1^2x_2 - z_1^2z_2 x_0^2 x_4),
 \end{equation}
where $x_1, \ldots, x_4$ are the Cox variables and $x_0 = 0$ is the hyperplane `at infinity' in the total coordinate space. If we set $r(\tau) = (r_1(\tau), r_2(\tau),r_3(\tau),r_4(\tau)) = (z_1,z_2,z_3,\tau)$, the variety $\overline{G \cdot r(\tau)}$ degenerates to $V_{\PP^4}(z_3x_1 - z_1x_3, x_0^2 x_4)$ for $\tau \rightarrow 0$, which is the union of a double plane at infinity and the plane $\overline{G \cdot  (z_1,z_2,z_3,0)}= V_{\PP^4}(z_3x_1 - z_1x_3, x_4)$. This means that if we slice \eqref{eq:orbiteq} with a general plane $L$ and we let $z_4 \rightarrow 0$, two out of three intersection points drift off to infinity and the other one lands on the orbit $G \cdot (z_1,z_2,z_3,0)$. 

Applying the same reasoning for $z_2 \rightarrow 0$, we see that \eqref{eq:orbiteq} degenerates to $V_{\PP^4}(z_3x_1 - z_1x_3,x_1^2 x_2)$, which is the union of the plane $V_{\PP^4}(x_1,x_3)$ with multiplicity 2 and the plane $\overline{G \cdot (z_1,0,z_3,z_4)} = V_{\PP^4}(z_3x_1 - z_1x_3, x_2)$. Note that the intersection of the first component $V_{\PP^4}(x_1,x_3)$ with $\CC^4$ is a component of the base locus $Z \subset \CC^4$. This means that if we slice \eqref{eq:orbiteq} with a general plane $L$ and we let $z_2 \rightarrow 0$, two out of three intersection points move towards the base locus and only one of them lands on the orbit $G \cdot (z_1,0,z_3,z_4)$.
\end{example}
Combining Lemma \ref{lem:representative}, Remark \ref{rem:power} and Theorem \ref{thm:boundary}, we see that if $\z_i(0) \in \pi((\CC^*)^\J)$, then a fraction of $\Vol(\Delta^\J)/\Vol(\Delta^{\{1,\ldots,k\}})$ of the representatives $z_{ij}(\tau)$ of $\z_i(\tau)$ will travel to $\pi^{-1}(\z_i(0))$ as $\tau \rightarrow 0$. This means that, although for the purpose of tracking paths in $T \subset X$ it suffices to track only one representative, at the very end of the tracking process we may have to switch representatives in order to find homogeneous coordinates of $\z_i(0)$. We propose to track only one representative per path in $X$ for $\tau \in [\tau^*, 1]$, where $\tau^*$ is `close' to 0, e. g. $\tau^* = 0.1$. At $\tau = \tau^*$, we initialize a specialized \alert{endgame} which tries to finish the path by switching representatives at $\tau = \tau^*$ until we land on a point in $\pi^{-1}(\z_i(0))$. We make this procedure explicit in Algorithm \ref{alg:endgame}.

\begin{algorithm}\caption{A specialized endgame for Cox homotopies}\label{alg:endgame}
\begin{algorithmic}[1]\Procedure{EndGame}{$(\calH,\L), 0 < \tau^* \leq 1, z \in V_{(\CC^*)^k}(\calH(x;\tau^*)) \cap L$}
\State found $\gets$ false
\While{ found $==$ false}
\State obtain $z_{\text{target}}$ by tracking $(\calH,\L)$ for $\tau \in [0,\tau^*]$ with starting solution  $z$ \label{line:tracking}
\If{$z_{\text{target}} < \infty$ and $z_{\text{target}} \notin Z$}
\State found $\gets$ true
\Else
\State $z \gets$ \textsc{SwitchRepresentative}$(z)$ \label{line:switchrep}
\EndIf
\EndWhile
\State \textbf{return} $z_{\text{target}}$ \Comment{A set of Cox coordinates for $\lim_{\tau \rightarrow 0^+} \z_i(\tau) \in V_X(\calF)$}
\EndProcedure
\end{algorithmic}
\end{algorithm}
A few comments are in order to clarify Algorithm \ref{alg:endgame}. First of all, note that Theorem \ref{thm:boundary} guarantees that the endgame terminates and that the output $z_{\text{target}}$ is such that $\pi(z_{\text{target}}) = \z_i(0)$ if $z = z_{ij}(\tau^*)$ for some $j$. In line \ref{line:tracking} we assume that the output of the tracking algorithm is $\infty$ in case the path diverges, and to check whether $z_{\text{target}} \in Z$ one can see if the residual with respect to the monomial generators of the irrelevant ideal $B$ is adequately small (using some sensible heuristic). In line \ref{line:switchrep}, the routine  \textsc{SwitchRepresentative} finds another representative $z'$ of $G \cdot z$ satisfying $\L(z') = 0$.
In our implementation, we may either enumerate representatives all at once, or dynamically by tracking monodromy loops on
$$
\L(z_1 \lambda^{P''_{:,1}}, \ldots, z_k  \lambda^{P''_{:,k}}) = \left ( \sum_{j = 1}^k A_{ij} z_j \lambda^{P''_{:,j}} + b_i \right)_{i = 1, \ldots, k-n} = 0
$$
with seed $\lambda_0 = (1,\ldots, 1)$ and setting $z' = (z_1 \lambda^{P''_{:,1}}, \ldots, z_k  \lambda^{P''_{:,k}})$ for any solution $\lambda \neq \lambda_0$. In practice one should check that $z'$ is a representative that has not been used before. 
In principal, though less straightforward to implement, the polyhedral homotopy could also be used in such an incremental strategy (cf.~\cite{mizutanimixedcell2007}).
Still, monodromy may be preferable if computing mixed cells is a bottleneck.
We refer to~\cite{duffmonodromy2019} for more details on monodromy and efficiency considerations.

\begin{remark}
In the case of (multi)projective homotopies, $\Vol(\Delta^\J) = 1/(k-n)!$ and $s_\J = 1$ are constant for all $\J$, meaning that in this case all representative paths will land on $\pi^{-1}(\z_i(0))$ for $\tau \rightarrow 0$. 
\end{remark}

\subsection{Solving equations on $X$}\label{subsec:algorithms}

Let $\f_1, \ldots, \f_n \in \CC[M]$ and let $\A = (\A_1, \ldots, \A_n)$ be the corresponding supports ($\A_i \subset M)$. In this subsection, we describe an algorithm for computing the solutions defined by $\calFhat = (\f_1, \ldots, \f_n) = 0$ on $X$, where $X = X_\Sigma$ is the $n$-dimensional toric variety coming from $\A$ as in Subsection \ref{subsec:homogenization}. That is, the algorithm computes all isolated points of $V_X(\calF)$, where $\calF$ is obtained from $\calFhat$ by homogenizing. Here, `computing' a point $\z \in X$ means computing numerical approximations of a set of Cox coordinates of $\z$ in $\CC^k \setminus Z$. We assume that we are given a system of affine start equations $\calGhat = (\g_1, \ldots, \g_n)$ such that $\g_i$ has support $\A_i$ and $|V_T(\calGhat)|$ consists of the BKK number $\D$ many starting solutions $\{\z_i(1)\}_{i=1, \ldots, \D}$, which are also given. As per usual, we denote $\calG$ for the homogenized start system. For a generic linear space $L = \L^{-1}(0)$ of $\CC^k \setminus Z$, we consider the Cox homotopy $(\calH,\L)$ where $\calH(x;\tau) = (1-\tau)\calG + \gamma \tau \calF$ and $\gamma \in \CC$ is either a random complex constant or $\gamma = 1$ when $\calG$ is sufficiently generic. Using insights from the previous sections, Algorithm \ref{alg:mainalg} is now immediate. 
\begin{algorithm}\caption{solve $\calFhat = 0$ on $X$}\label{alg:mainalg}
\begin{algorithmic}[1]\Procedure{SolveViaCoxHomotopy}{$\calFhat, \calGhat, \{\z_1(1), \ldots, \z_\D(1) \}, \tau_{\text{EG}}$}
\State $\calF, \calG \gets$ homogenize $\calFhat, \calGhat$
\State $\calH \gets  (1-\tau)\calG + \gamma \tau \calF$
\State $\L \gets$ random affine map $Ax + b$ \label{line:L}
\State $\{z_1(1),\ldots, z_\D(1) \} \gets$ \textsc{HomogenizeStartingsolutions($\{\z_1(1), \ldots, \z_\D(1) \}, \calG, \L$)} \label{line:homog}
\State $\{ z_1(\tau_{\text{EG}}), \ldots, z_\D(\tau_{\text{EG}}) \} \gets$ track $\{z_1(1),\ldots, z_\D(1) \}$ along $(\calH, \L)$ for $\tau \in [\tau_{\text{EG}},1]$ \label{line:track2}
\For{$i = 1, \dots, \D$}
\State $z_{i,\text{target}} \gets$ \textsc{EndGame}($(\calH,\L),\tau_{\text{EG}},z_i(\tau_{\text{EG}}))$) \label{line:endgame}
\EndFor
\State \textbf{return} $\{z_{1,\text{target}}, \ldots, z_{\D,\text{target}} \}$ \Comment{A set of Cox coordinates for each point in $V_X(\calF)$}
\EndProcedure
\end{algorithmic}
\end{algorithm}
In line \ref{line:homog} of this algorithm, the given starting solutions $\{\z_1(1), \ldots, \z_\D(1) \}$ are lifted to the points $\{z_1(1),\ldots, z_\D(1) \}$ in the total coordinate space in such a way that $\pi(z_i(1)) = \z_i(1)$ and $\L(z_i(1)) = 0, i = 1, \ldots, \D$. This is done using Algorithm \ref{alg:homogenizestart}, which we discuss below. 
In line \ref{line:track2} of Algorithm \ref{alg:mainalg}, the homogenized starting solutions are tracked for $\tau$ going from 1 to $\tau_{\text{EG}}$, which is a parameter indicating where the \alert{endgame operating region} $ \tau \in [0, \tau_{\text{EG}})$ starts. We propose $\tau_{\text{EG}} = 0.1$ as a default value. The tracking can happen in only $n$ instead of $k$ variables, by setting $x = \hat{x} + K y$, where $\hat{x} \in \CC^k$ is any point in $L$ and $K$ is a matrix whose columns span $\ker A$, where $A$ is the matrix from line \ref{line:L}. Line \ref{line:endgame} uses Algorithm \ref{alg:endgame}. \\

We now discuss what happens in line \ref{line:homog} in more detail. First, a set of points $\tilde{z}_i \in (\CC^*)^k$ satisfying $\pi(\tilde{z}_i(1)) = \z_i(1)$ is computed. Since $\z_i(1) = (t_{i1}, \ldots, t_{in}) \in T = (\CC^*)^n, i = 1, \ldots, \D$,we must have $\tilde{z}_i(1)^{F_{j,:}} = t_{ij}, j = 1, \ldots, n$ by \eqref{eq:monmap}. Writing $(v_{i1}, \ldots, v_{ik})$ for the (unknown) coordinates of $\tilde{z}_i(1)$ on $(\CC^*)^k$, we get the system of binomial equations
\begin{equation} \label{eq:binomeq}
v_{i1}^{F_{j,1}} v_{i2}^{F_{j,2}} \cdots v_{ik}^{F_{j,k}} = t_{ij}, \quad i = 1, \ldots, \D, \quad  j = 1, \ldots, \D.
\end{equation}
Taking $\log(\cdot)$ on both sides (using any choice of branch) gives 
\begin{equation} \label{eq:logeq}
F v_{\log} = t_{\log},
\end{equation}
where $F$ is the facet matrix, $(v_{\log})_{ij} = \log v_{ji}$ and $(t_{\log})_{ij} = \log t_{ji}$. It is clear that a solution $v_{\log}$ of the linear equations \eqref{eq:logeq} gives a solution $v_{ij} = \exp((v_{\log})_{ji})$ to \eqref{eq:binomeq}. Let $\tilde{F} = F_{:,\{i_1, \ldots, i_n\}} = [u_{i_1} ~ \cdots ~ u_{i_n}]$ be an invertible submatrix of $F$, consisting of the ray generators indexed by $\{i_1, \ldots, i_n\}$. Such a matrix $\tilde{F}$ exists, since $\Sigma$ is complete. A solution to \eqref{eq:logeq} is given by 
$$(v_{\log})_{\{i_1, \ldots, i_n\},:} = \tilde{F}^{-1} t_{\log}, \qquad (v_{\log})_{\{1, \ldots, k\} \setminus \{i_1, \ldots, i_n\},:} = 0.$$
In order to reduce rounding errors in the computation of $\tilde{F}^{-1} t_{\log}$, it is favorable to pick a \alert{well-conditioned} submatrix $\tilde{F}$. This can be done, for instance, using a strong rank-revealing QR factorization \cite{gu1996efficient}. The obtained solutions $\tilde{z}_i(1)$ satisfy $\pi(\tilde{z}_i(1)) = \z_i(1)$, and hence $\calG(\tilde{z}_i(1)) = 0$. It remains to track the $\tilde{z}_i(1)$ through the $G$-orbit $\pi^{-1}(\z_i(1))$ to obtain $z_i(1)$ satisfying both $\calG(z_i(1)) = 0$ and $\L(z_i(1)) = 0$. For that, note that the points $\tilde{z}_i(1)$ satisfy the (very non-generic) linear conditions $\tilde{z}_i(1) \in \L_1^{-1}(0)$, where $\L_1(x) = (x_i - 1)_{i \in \{1, \ldots, k\} \setminus \{i_1, \ldots, i_n\} }$. Therefore, we can track the homotopy $(\calG, \gamma \tau \L_1 + (1-\tau) \L)$ which intersects $V_{\CC^k \setminus Z}(\calG)$ with a moving linear space for $\tau$ going from 1 to 0, with starting solutions $\tilde{z}_i(1)$. 
Homotopies with a moving linear space are fundamental in numerical algebraic geometry, particularly in homotopy membership testing and monodromy~\cite[\S 15.4]{SW05}.
We summarize this discussion in Algorithm \ref{alg:homogenizestart}.
\begin{algorithm}\caption{Lift a set of affine starting solutions in $T$ to $(\CC^*)^k$}\label{alg:homogenizestart}
\begin{algorithmic}[1]\Procedure{HomogenizeStartingSolutions}{$\{\z_1(1), \ldots, \z_\D(1) \},\calG, \L$}
\State $\{i_1, \ldots, i_n\} \gets $ a subset of indices in $\{1,\ldots, k\}$ such that $\tilde{F} = F_{:, \{ i_1, \ldots, i_n \}}$ is invertible
\State $\{\tilde{z}_1(1), \ldots, \tilde{z}_\D(1) \} \gets $ a solution of \eqref{eq:binomeq} obtained via \eqref{eq:logeq}
\State $\L_1(x) \gets (x_i - 1)_{i \in \{1, \ldots, k\} \setminus \{i_1, \ldots, i_n\} }$
\State $\LL(x;\tau) \gets  \gamma \tau \L_1(x) + (\tau - 1) \L(x)$
\State $\{z_1(1), \ldots, z_\D(1) \} \gets$ track $\{\tilde{z}_1(1), \ldots, \tilde{z}_\D(1) \}$ along $(\calG, \LL(x;\tau))$ for $\tau \in [0,1]$
\State \textbf{return} $\{z_1(1), \ldots, z_\D(1) \} $ \Comment{A set of points in $V_{\CC^k \setminus Z}(\calG, \L)$}
\EndProcedure
\end{algorithmic}
\end{algorithm}

\subsection{Orthogonal slicing in $\CC^k \setminus Z$} \label{subsec:orthslice}

In the case where $X = \PP^n$, the linear part $\L$ of the Cox homotopy $(\calH,\L)$ is a single linear equation in the $x$-variables ($k = n+1$ and the orbits have dimension 1). As pointed out in Remark \ref{rem:movinglinspace}, we may let $\L(x;\tau)$ depend on the continuation parameter $\tau$. The points in $\CC^{n+1} \setminus \{0\}$ satisfying $\L(x;\tau) = 0$ for $\tau \in (0,1]$ lie on a moving hyperplane. This can be thought of as a continuously varying \alert{affine patch} in which the homotopy is being tracked. In \cite{hauenstein2018adaptive}, the authors propose several adaptive strategies for choosing this patch. One natural choice they propose is that of an \alert{orthogonal patch} (see \cite[Subsection 3.2]{hauenstein2018adaptive}). In this subsection, we discuss how this can be done quite naturally in the total coordinate space of any compact toric variety $X$. \\ 
Let $z \in (\CC^*)^k$ and consider the corresponding orbit $G \cdot z$. Locally (and if $\Cl(X)$ is free, even globally), this orbit is parametrized by 
$$( \lambda^{P_{:,1}''} z_1 , \ldots, \lambda^{P_{:,k}''}z_k ), \quad \lambda \in \CC^{k-n},$$
where $P''$ comes from the Smith normal form of $F^\top$, see Lemma \ref{lem:orbitparam}. The tangent space to the orbit at $( \lambda^{P_{:,1}''} z_1 , \ldots, \lambda^{P_{:,k}''}z_k )$ is parametrized by 
$$ z + \sum_{i=1}^{n-k} c_i \frac{\partial }{\partial \lambda_i}  ( \lambda^{P_{:,1}''} z_1 , \ldots, \lambda^{P_{:,k}''}z_k ) , \quad c_i \in \CC,$$
where we view $z$ as a row vector of length $k$. 
For $\lambda = (1,\ldots, 1)$, we find that the tangent space to the orbit at $z$ is given by the simple expression
$$ z + c^\top P'' \text{diag}(z_1,\ldots,z_k) , \quad c \in \CC^{k-n},$$
where $\text{diag}(z_1,\ldots,z_k)$ is a diagonal $k \times k$ matrix with the coordinates of $z$ on its diagonal. It follows that $x$ is in the normal space to the orbit $G \cdot z$ at $z$ if and only if $x-z$ is orthogonal to the rows of $P'' \text{diag}(z_1,\ldots,z_k)$. \\
For a smooth path $(z(\tau),\tau) \in (\CC^k \setminus Z) \times (0,1]$ satisfying $\calH(z(\tau),\tau) = 0$, we define
$$ \L(x;\tau) = \text{conj}(P'' \text{diag}(z_1(\tau),\ldots,z_k(\tau))) (x - z(\tau)),$$
where $\text{conj}(\cdot)$ takes the (entry-wise) complex conjugate. This suggests Algorithm \ref{alg:orthslice} for tracking one path in the total coordinate space of $X$, using orthogonal slicing for collecting representatives on the orbits. 
\begin{algorithm}\caption{Track one path in the total coordinate space using orthogonal slicing}\label{alg:orthslice}
\begin{algorithmic}[1]\Procedure{TrackOrth}{$\calG,\calF,z(1),P''$}
\State $z \gets z(1)$\Comment{Starting solution: $\calG(z(1)) = 0$}
\State $\calH(x;\tau) \gets (1-\tau)\calG(x) + \gamma \tau \calF(x)$
\State $\tau^* \gets 1$
\While{$\tau^* > 0$}\Comment{The target parameter value for $\tau$ is 0 }
\State $\L(x;\tau^*) \gets \text{conj}(P'' \text{diag}(z_1, \ldots, z_k))(x - z)$
\label{step:orthpatch} 
\State $(\tilde{z}, \Delta \tau) \gets \textsc{Predict}((\calH,\L(x;\tau^*)),z,\tau^*)$ \label{step:predict} \Comment{Adaptive stepsize predictor routine}
\State $ z \gets \textsc{Correct}((\calH,\L),\tilde{z}, \tau^* - \Delta \tau)$ \label{step:correct} \Comment{Corrector routine, e.g.\ Newton iteration}
\State $ \tau^* \gets \tau^* - \Delta \tau$
\EndWhile
\State \textbf{return} $z$
\Comment{$z$ is the target solution.}
\EndProcedure
\end{algorithmic}
\end{algorithm}
The algorithm uses the blackbox functions \textsc{Predict} (line \ref{step:predict}) and \textsc{Correct} (line \ref{step:correct}) which are assumed to implement a predictor-corrector path tracking scheme, possibly (and preferably) using an adaptive step size control \cite{schwetlick1987higher,kearfott1994interval,dedieu2013adaptive,timme2019adaptive,telen2019robust}. The function \textsc{Predict} returns a point $\tilde{z}$ and a step size $\Delta \tau$ such that $\tilde{z}$ is an approximation for a solution of $(\calH(x;\tau^*-\Delta \tau), \L(x;\tau^*))$ and $\Delta \tau$ is a `safe' step size. The function \textsc{Correct} then refines $\tilde{z}$ using, for instance, Newton iteration on $(\calH(x;\tau^*-\Delta \tau), \L(x;\tau^*))$ with starting point $\tilde{z}$.

\section{Numerical examples} \label{sec:numexps}
We have implemented the algorithms presented in Section \ref{sec:mainthm} in Julia, making use of the packages Polymake.jl~\cite{KLT20} and HomotopyContinuation.jl \cite{breiding2018homotopycontinuation}. We present a selection of experiments highlighting the advantages of Cox homotopies, as discussed in the introduction. More material and the code will be made available on \url{https://mathrepo.mis.mpg.de}.
\begin{experiment}[A problem from computer vision] \label{exp:computervision}

The $8$-point problem for cameras with radial distortion~\cite{kukelova2011minimal} consists of $9$ equations in $9$ unknowns. 
Eight equations are given by
\begin{equation}
\label{eq:epipole}
\begin{pmatrix}
  p_{i,1}' & p_{i,2}' & 1 + r_i' \lambda
\end{pmatrix}
\underbrace{
  \begin{pmatrix}
    f_{1,1} & f_{1,2} & f_{1,3}\\
    f_{2,1} & f_{2,2} & f_{2,3}\\
    f_{3,1} & f_{3,2} & 1
  \end{pmatrix}
}_{\boldsymbol{F}}
\begin{pmatrix}
  p_{i,1} \\
  p_{i,2} \\
  1 + r_i \lambda
\end{pmatrix} = 0,
  \end{equation}
where the parameters $p_{i,j}, p_{i,j}'$ and $r_i, r_i'$ are known and represent distorted image coordinates and distortion radii, respectively. 
The true image coordinates are known only after the \alert{radial distortion parameter} $\lambda $ is recovered.
The matrix $\boldsymbol{F}$ is called the \alert{fundamental matrix} and satisfies the additional constraint \begin{equation}
\label{eq:detF}
\det \boldsymbol{F} = 0.
\end{equation}
For more details on the model and problem, we refer to~\cite{fitzgibbon2001simultaneous,kukelova2011minimal}.

For generic parameters $p_{i,j}, p_{i,j}'$ and $r_i, r_i',$ , the number of solutions to equations~\eqref{eq:epipole} and~\eqref{eq:detF} is exactly the BKK number $\D = 16.$
For comparison, the B\'ezout bound is $768.$
Homogenizing these equations as in Section~\ref{subsec:homogenization}, there are $26$ Cox coordinates.
In our experiment, we consider \emph{nearly degenerate systems} satisfying $q_{i,1}' = -r_{i,1}' +\epsilon $ for $\epsilon $ small.
For $\epsilon = 10^{-7},$ the \texttt{solve} command in HomotopyContinuation.jl reports $8$ solutions and $8$ paths going towards infinity.
Our implementation of Algorithm~\ref{alg:mainalg} finds a representative in $\CC^{26}$ for each endpoint of these 16 paths in the compact toric variety $X_\Sigma .$

Since the generic orbit degree $d=4583$ appearing in Theorem~\ref{thm:coxhomotopy} is quite large, we may dynamically enumerate the representatives of each path considered in Algorithm~\ref{alg:endgame} using random monodromy loops.
This proves to be far more efficient than tracking a total of $d \D = 73328$ paths.
Ignorging the pre-computation of the polytope $\Pol$ and facet matrix $F^\top,$ the entire procedure takes around 15--45s between several runs on the same machine.
Most time is spent on the endgame in Algorithm \ref{alg:endgame}.
We observe that the $8$ nearly-infinite, nearly-singular solutions in the torus $T\subset \CC^9$ are reasonably accurate in the sense of backward error.
More precisely, the residuals used in~\cite{telen2020thesis} for the $8$ regular solutions are all on the order of unit roundoff $\approx 10^{-16},$ and for the other $8$ are in the range from $10^{-7}$ to $10^{-15}.$ 
\end{experiment}
\begin{experiment}[Intersecting curves on a Hirzebruch surface] \label{exp:hirzebruch}
Consider again the Hirzebruch surface $X = \HH_2$ from Example \ref{ex:hirzebruch}. Its fan $\Sigma$ and facet matrix $F$ are shown in Figure \ref{fig:hirzebruch}. In this experiment, we illustrate the use of \alert{orthogonal slicing} (see Subsection \ref{subsec:orthslice}) as an adaptive strategy for tracking paths in a Cox homotopy. For this, we generate Laurent polynomials $\f_1, \f_2 \in \CC[t_1^{\pm 1}, t_2^{\pm 1}]$ with support $\A_1 = \A_2 = \Pol_i \cap M$, where $\Pol_i$ has facet representation given by $F^\top$ and $a_i = (0,0,10,10)^\top$ for $i = 1,2$. These Laurent polynomials are random in the sense that the coefficients of the $\f_i$ 
have been drawn from a standard normal distribution. There are 400 solutions to $\f_1 = \f_2 = 0$ in $T \subset X$.
For illustration purposes, we use an implementation of Algorithm \ref{alg:orthslice} with the naive predictor 
\begin{center}
\textsc{Predict}($(\calH,\L(x;\tau^*)), z, \tau^*) = (z, 0.001)$.
\end{center}
We choose three of the 256 paths $z(\tau)$ in the Cox homotopy randomly and keep track of the condition number of the Jacobian $J_{\tau^*} = \left. \partial_x (\calH(x;\tau^*), \L(x;\tau^*)) \right|_{x = z(\tau^*)}$ for $\tau^* = 1, 1- \Delta \tau, 1- 2 \Delta \tau , \ldots, 0$, with $\Delta \tau = 0.001$. For comparison, we track the same paths in $X$ using 10 random slices $\L_{\text{\text{rand}}}(x)$ given by $\L_{\text{\text{rand}}}(x) = A(x - z(1))$ where $A$ is a matrix with entries drawn from a complex standard normal distribution. The results are illustrated in Figure \ref{fig:hirzebruchorth}. The dashed (orange) curves are obtained by taking the geometric mean of the dotted (grey) curves, which represent the condition number of $J_\tau$ for 10 random slices. The blue curves represent the condition number of $J_\tau$ using Algorithm \ref{alg:orthslice}.
\begin{figure}
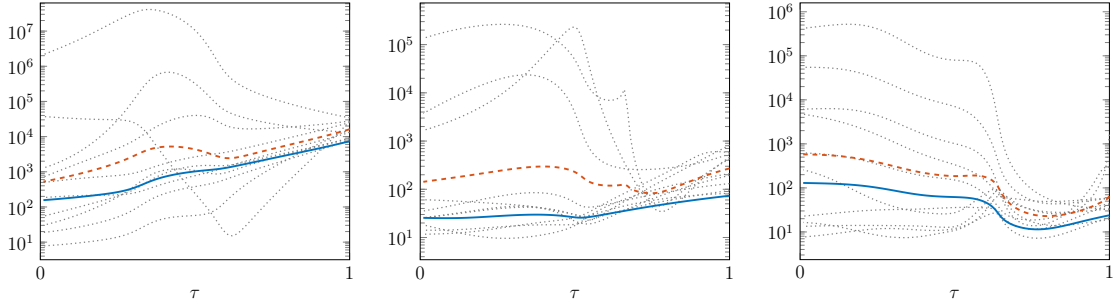

\input{orthslice1.tex}
\input{orthslice2.tex}
\input{orthslice3.tex}
\caption{Condition number of the Jacobian along 3 paths of the homotopy in Experiment \ref{exp:hirzebruch} using the orthogonal slicing strategy of Algorithm \ref{alg:orthslice} (\ref{bluecurve}) and the average condition number (\ref{orangedashedcurve}) for 10 random linear slices (\ref{greycurve}).}
\label{fig:hirzebruchorth}
\end{figure}
The figure shows that for different randomly generated $\L_{\text{\text{rand}}}(x)$, the behavior of the condition number may vary significantly. Using orthogonal slicing, we consistently obtain smaller condition numbers on average. Moreover, the computation of the orthogonal slice causes virtually no computational overhead. 
\end{experiment}

\begin{experiment}[Solving equations on a weighted projective space] \label{exp:weighted}
Consider the system of polynomial equations 
\begin{align*}
\f_1 = (3 + \varepsilon_1)t_1^2 + 7 t_1 t_2 + 7 t_2^2 + 9t_1  + 3 t_2 + 9 t_3 + 2 &= 0,\\
\f_2 = (3 + \varepsilon_2)t_1^2 + 7 t_1 t_2 + 7 t_2^2 + 5t_1 + 2 t_2 + 3t_3 + 4 &= 0, \\
\f_3 = (3 + \varepsilon_3)t_1^2 + 7 t_1 t_2 + 7 t_2^2 + 4t_1 + 8 t_2 + 4t_3 + 9 &= 0,
\end{align*}
in the variables $t_1,t_2,t_3$. The parameters $\varepsilon_i$ are assigned random complex values of modulus $10^{-12}$. The normalized volume of the Newton polytope of these equations is 4, which is equal to the number of solutions in $T = (\CC^*)^3$. However, using the command {\tt solve} in HomotopyContinuation.jl, we find only 2 solutions. The polynomials $\f_1, \f_2, \f_3$ homogenize to degree 2 elements in the Cox ring $S = \CC[x_1,x_2,x_3,x_4]$ of the weighted projective threefold $X = \PP_{1,1,2,1}$, where $x_i$ corresponds to the facet normal $e_i$ for $i = 1, 2, 3$ and $\{x_4 = 0\}$ is the divisor `at infinity'. Using Algorithm \ref{alg:mainalg}, we find all 4 solutions in $X$. We observe that the Cox coordinates $x_3$ and $x_4$ have absolute value $\approx 10^{-12}$ for two of these solutions, which means they lie close to $D_3 \cap D_4$. The corresponding points in the torus have coordinates of modulus $\approx 10^{12}$. Tracking these solutions in $(\CC^*)^3$ causes premature truncation in the standard polyhedral homotopy. 

We now test our approach on a larger example with similar behavior. The system is given by $\f_1 = \cdots = \f_5 = 0$, where the $\f_i$ have Newton polytope $P = \{ m \in \RR^5 ~|~ F^\top m + a \geq 0 \}$, with 
$$ F^\top = \begin{bmatrix}
1 & 0 & 0 & 0 & 0 \\
0 & 1 & 0 & 0 & 0 \\ 
0 & 0 & 1 & 0 & 0 \\ 
0 & 0 & 0 & 1 & 0 \\ 
0 & 0 & 0 & 0 & 1 \\
-1 & -2 & -2 & -2 & -4
\end{bmatrix}, \qquad a = \begin{bmatrix}
0\\
0\\ 
0\\
0\\
0\\
12
\end{bmatrix}.$$
These correspond to degree 12 equations with 7776 solutions in the weighted projective space $X = \PP_{1,2,2,2,4,1}$. We manipulate the coefficients such that 216 of these solutions lie near $D_5 \cap D_6$, meaning that their Cox coordinates $x_5$ and $x_6$ have modulus $\approx \varepsilon$. For $\varepsilon = 10^{-7}$, the polyhedral homotopy misses about $200$ solutions. Algorithm \ref{alg:mainalg} finds all 7776 solutions in $T$ within 7 minutes and 40 seconds on a 16 GB MacBook Pro machine with an Intel Core i7 processor working at 2.6 GHz.
\end{experiment}

\begin{experiment}[Equations on a Bott-Samelson variety] \label{exp:BS} 
Defining $B \subseteq GL_3$ as the subgroup of upper-triangular matrices, $GL_3/B$ is birational to a Bott-Samelson variety. We draw a random complex square polynomial system on $GL_3/B$ from the Khovanskii basis $\mathcal{B} =  \{1,x,y,z,xz,yz,x(xz+y),y(xz+y)\}$~\cite{Anderson}. That is, we consider the system $$\hat{f}_i = c_{i,1}\;1 + c_{i,2}\;x + c_{i,3}\;y + c_{i,4}\;z + c_{i,5}\;xz + c_{i,6}\;yz + c_{i,7}\;x^2z + c_{i,8}\; xyz + c_{i,7}\;xy + c_{i,8}\;y^2 ,$$
where $i \in \{1,2,3\}$ and $ c_{i,j} \in \CC$ for $i = 1,2,3$ and $j = 1, \dots, 8 $. 
\begin{figure}
\tdplotsetmaincoords{70}{110}
\begin{tikzpicture}[baseline=(A.base),scale=1.5,tdplot_main_coords]
\coordinate (A) at (0,0,0);
\coordinate (B) at (1,0,0);
\coordinate (C) at (2,0,1);
\coordinate (D) at (0,0,1);
\coordinate (E) at (0,2,0);
\coordinate (F) at (0,1,1);
\coordinate (G) at (1,1,1);
\coordinate (H) at (1,1,0);
			
\draw[fill opacity=0.2,fill = mycolor2,] (A)--(B)--(C)--(D)--cycle;
\draw[fill opacity=0.6,fill = mycolor2,] (A)--(D)--(F)--(E)--cycle;
\draw[fill opacity=0.6,fill = mycolor2,] (A)--(E)--(H)--(B)--cycle;
\draw[fill opacity=0.6,fill = mycolor1,] (G)--(C)--(H)--(E)--cycle; 
\draw[fill opacity=0.6,fill = mycolor2,] (C)--(B)--(H)--cycle;
\draw[fill opacity=0.6,fill = mycolor2,] (F)--(G)--(E)--cycle;
\draw[fill opacity=0.6,fill = mycolor2,] (D)--(C)--(G)--(F)--cycle;
\end{tikzpicture} 
\hspace{1.5cm}
\tdplotsetmaincoords{70}{110}
\begin{tikzpicture}[baseline=(A.base),scale=1.5,tdplot_main_coords]
\coordinate (A) at (0,0,0);
\coordinate (B) at (1,0,0);
\coordinate (C) at (1,0,1); 
\coordinate (D) at (0,0,1); 
\coordinate (E) at (0,2,0); 
\coordinate (F) at (0,1,1); 
\coordinate (G) at (1,1,1); 
\coordinate (H) at (1,1,0);
			
\draw[fill opacity=0.2,fill = mycolor2,] (A)--(B)--(C)--(D)--cycle;
\draw[fill opacity=0.6,fill = mycolor2,] (A)--(D)--(F)--(E)--cycle;
\draw[fill opacity=0.6,fill = mycolor2,] (A)--(E)--(H)--(B)--cycle;
\draw[fill opacity=0.6,fill = mycolor2,] (F)--(C)--(H)--(E)--cycle; 
\draw[fill opacity=0.6,fill = mycolor2,] (C)--(B)--(H)--cycle;
\draw[fill opacity=0.6,fill = mycolor2,] (F)--(C)--(D)--cycle; 
\end{tikzpicture}
\caption{The Newton polytope $\Pol_i$ (left) of the equations in Experiment \ref{exp:BS} and the Newton-Okounkov body associated to $\mathcal{B}$ (right). The face of $\Pol_i$ whose corresponding system has solutions at infinity is highlighted in blue.}
\label{fig:BS}
\end{figure}
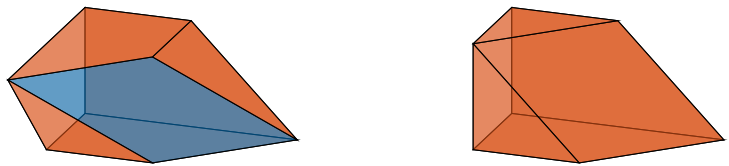
The BKK bound for $\hat{\mathcal{F}} = (\hat{f}_1, \hat{f}_2, \hat{f}_3)$ is $10$, which is given the by the normalized volume of the Newton polytope $\mathscr{P}_i$ of $\hat{f_i}$. The number of solutions to $\hat{\mathcal{F}}=0$, however, is known to be equal to the normalized volume of the Newton-Okounkov body associated to $\mathcal{B}$, which is six. Both the Newton polytope $\mathscr{P}$ and the Newton-Okounkov body associated to $\mathcal{B}$ are depicted in Figure~\ref{fig:BS}. 
This deficient root count (six) with respect to the BKK bound (ten) suggests that there are solutions at infinity, which would correspond to solutions of face system(s) of $\hat{\mathcal{F}}$. We can find solutions to face systems via the Cox homotopy.

We homogenize $\hat{\mathcal{F}}$ using the normal fan of $\mathscr{P}$, whose rays are recorded by the columns of
 \[F = \begin{bmatrix} -1 &  0 & 0 & -1 & 1 & 0  & 0  \\ 
                                 -1 &  0 & 1 &  0 & 0 &  0 & -1  \\ 
                                  0 & -1 & 0 &  1 & 0 &  1 & -1\end{bmatrix}.\]
Using the Cox homotopy, we find that the homogenization of $\hat{\mathcal{F}}$, $\mathcal{F}$, has six nonsingular solutions in the torus and four singular solutions whose first Cox coordinate is zero. This indicates that there are solutions to the face system of $\hat{\mathcal{F}}$ associated to the ray $(-1, -1, 0)$. This face system, which is given by
$c_{i,7}\; x(xz+y) + c_{i,8}\; y(xz+y) = 0$  for $i  = 1, 2, 3$, has infinitely many solutions along the curve $C$ defined by $y = -xz$. Therefore, the root deficiency of $\hat{\mathcal{F}}$ is explained by this curve at infinity. 
The orbit degree $\deg(\overline{G\cdot z})$ in this example is five for $z \in (\CC^*)^k$. The degree $\deg(\overline{G \cdot z})$ for $z \in \pi^{-1}(\z)$ drops to three for general points on $D_1 = V_X(x_1)$. Using the specialized endgame in Algorithm \ref{alg:endgame}, we consistently find the homogeneous coordinates of four points on $C$. Although Lemma \ref{lem:representative} only accounts for isolated points in $V_X(\calF)$, this example shows that the Cox homotopy can also detect positive dimensional components on $X \setminus T$. This strategy provides an important first step towards generalizing \cite{bates2019excess} for computing \alert{numerical irreducible decompositions} in $X$.
\end{experiment}

\section{Conclusion}
We have introduced Cox homotopies for tracking solution paths of sparse polynomial systems in a compact toric variety $X$. The algorithm makes explicit use of the construction of $X$ as a GIT quotient of a quasi-affine space by the action of a reductive group $G$. We have described the degree of (the closure of) $G$-orbits in this construction in terms of volumes of orbit polytopes, lattice indices, and invariant factors. As we have shown in our experiments, Cox homotopies, as a generalization of (multi)projective homotopies, allow us to deal with solutions on or near the boundary of the dense torus in $X$ in an elegant way, avoiding premature truncation of solution paths and providing insight in the solution structure `at infinity'. It inherits the advantage of polyhedral homotopies that only the BKK number many paths need to be tracked. Experiments show that our algorithms provide the first steps towards performing numerical irreducible decomposition in $X$. 
\section*{Acknowledgements}
The first steps in this research project were taken at the conference Ideals, Varieties and Applications, organized in June 2019 to honor David Cox and his influence in several areas of Commutative Algebra and Algebraic Geometry. Part of the research was conducted during the MATH+ Thematic Einstein Semester on Algebraic Geometry, Varieties, Polyhedra, Computation in Berlin. We want to thank the organizers of these meetings for making this collaboration possible. We also want to thank the Max Planck Institute for Mathematics in the Sciences for its resources used in aiding this collaboration. We are grateful to Sascha Timme for his valuable help with Julia and HomotopyContinuation.jl and to Frank Sottile for fruitful discussions and useful suggestions.
\bibliographystyle{abbrv}
\bibliography{bibl}

\end{document}